\documentclass[12pt]{amsart}
\usepackage{amsmath,amssymb,amsfonts}
\usepackage{mathtools}
\usepackage{esvect}
\usepackage{tikz}
\usepackage[shortlabels]{enumitem}
\setlist{leftmargin=*}

\usepackage{amsfonts}
\usepackage{todonotes}

\usepackage[colorlinks=true, linkcolor=blue]{hyperref}

\usepackage[sorted]{amsrefs}

\newtheorem{thm}{Theorem}[section]
\newtheorem*{thm*}{Theorem}
\newtheorem{cor}[thm]{Corollary}
\newtheorem{lem}[thm]{Lemma}

\newtheorem{fact}[thm]{Fact}

\theoremstyle{definition}
\newtheorem{defn}[thm]{Definition}
\theoremstyle{remark}

\newtheorem{rem}[thm]{Remark}

 \numberwithin{equation}{section}
    {\medskip\begingroup\leftskip 0.5cm\rightskip 0.5cm\noindent\begin{small}{\bf Remark.}}
    {\end{small}\par\endgroup}
{\begin{list}{$\bullet$}
 {\settowidth{\labelwidth}{\textsf{$\bullet$}} \setlength{\leftmargin}{10pt}}}
{\end{list}}

\newcounter{ssample}[section]

{\color{Bittersweet} \noindent Example \refstepcounter{ssample}\hbox{\bf \arabic{section}.\arabic{ssample}.}}
{}

\newcounter{insertcount}

    {\color{blue} \medskip\begingroup\noindent\begin{small}{\color{blue} \stepcounter{insertcount}
          {
            \bf Insert \arabic{insertcount}. #1.}
            \addcontentsline{toc}{subsection}{{\ \ \small  Insert \arabic{insertcount}: #1}}
               \leavevmode  }
           }
    {\end{small}\par\endgroup}

\newcommand{\mrmk}[1]
{{\tiny$^{\spadesuit}$}\marginpar{\fbox{\footnotesize #1}}}
\def\strutdepth{\dp\strutbox}%
\def\marginalnote#1{\strut\vadjust{\kern-\strutdepth\specialnote{#1}}}%
\def\specialnote#1{\vtop to \strutdepth{\baselineskip%
\strutdepth\vss\llap{\hbox{\scriptsize \bf #1}}\null}}%

\newcommand{\RR}{\mathbb{R}}

\def\vc{\operatorname{vc}}

\def\CF{\mathcal F}

\newcommand{\NN}{\mathbb N}

\newcommand{\CI}{\mathcal I}

\reversemarginpar

\def\CL{\mathcal{L}}

\def\CM{\mathcal{M}}

\newcommand*\bbar[1]{%
  \hbox{%
    \vbox{%
      \hrule height 0.5pt 
      \kern0.5ex
      \hbox{%
        \kern-0.1em
        \ensuremath{#1}%
        \kern-0.1em
      }%
    }%
  }%
}

\gdef\CT{{\mathcal{T}}}
\gdef\ttimes{{\times}}
\gdef\ddotsb{{\dotsb}}
\title[Cutting lemma in distal structures]{Cutting lemma and Zarankiewicz's problem in distal structures}
 
\author{Artem Chernikov} \address{Department of Mathematics, 
University of California Los Angeles,
Los Angeles, CA 90095-1555} \email{chernikov@math.ucla.edu}

\author{David Galvin} \address{Department of Mathematics, University of Notre Dame, Notre Dame,
  IN 46556} \email{dgalvin1@nd.edu}

\author{Sergei Starchenko} \address{Department of Mathematics, University of Notre Dame, Notre Dame,
  IN 46556} \email{Starchenko.1@nd.edu}

\begin{document}

{\abstract We establish a cutting lemma for definable families of sets in distal structures, as well as the optimality of the distal cell decomposition for definable families of sets on the plane in o-minimal expansions of fields. Using it, we generalize the results in \cite{zaran} on the semialgebraic planar Zarankiewicz problem to arbitrary o-minimal structures, in particular obtaining an o-minimal generalization of the Szemer\'edi-Trotter theorem.}

\maketitle

\section{Introduction}

The so called \emph{cutting lemma} is a very useful combinatorial partition tool with numerous applications in computational and incidence geometry and related areas (see e.g. \cite[Sections 4.5, 6.5]{mbook} or \cite{cuttings} for a survey). In its simplest form it can be stated as follows (see e.g. \cite[Lemma 4.5.3]{mbook}).

\begin{fact}
For every set $L$ of $n$ lines in the real plane and every $1< r < n$ there exists a $\frac{1}{r}$-cutting for $L$ of size $O(r^2)$. That is, there is a subdivision	of the plane into generalized triangles (i.e. intersections of three half-planes)
 $\Delta_1, \ldots, \Delta_t$ so that the interior of each $\Delta_i$ is intersected by at most $\frac{n}{r}$ lines in $L$, and we have $t \leq C r^2$ for a certain constant $C$ independent of $n$ and $r$.\end{fact}

This result provides a method to analyze intersection patterns in families of lines, and it has many generalizations to higher dimensional sets and/or to families of sets of more complicated shape than lines, for example for families of algebraic or semialgebraic curves of bounded complexity \cite{chaz}. The proofs of these generalizations  typically combine some kind of geometric ``cell decomposition'' result with the so-called random sampling technique of Clarkson and Shor \cite{13}. 

The aim of this article is to establish a general version of the cutting lemma for definable (in the sense of first-order logic) families of sets in a certain model-theoretically tame class of structures (namely, for distal structures --- see Section \ref{sec: prelims and distality} for the definition), as well as to apply it to generalize some of the results in the area from the semialgebraic context  to arbitrary o-minimal structures. This work can be viewed as a continuation and refinement of the work started in \cite{distal}, where the connection of model-theoretic distality with a weak form of the cutting lemma was discovered (we don't assume familiarity with that paper, but recommend its introduction for an expanded discussion of the model theoretic preliminaries).  We believe that distal structures provide the most general natural setting for investigating questions in ``generalized incidence combinatorics''.

Let us describe the main results of the paper. Our first theorem establishes a cutting lemma for a definable family of sets in a distal structure, with the bound corresponding to the bound on the size of its distal cell decomposition. This can be viewed as  a generalized form of Matou\v{s}ek's axiomatic treatment of Clarkson's random sampling method discussed in \cite[Section 6.5]{mbook}.
The proof relies in particular on Lemma \ref{lem-set_system_corr} on correlations in set-systems to deal with the lack of the corresponding notion of ``being in a general position''.
\begin{thm*}(Theorem \ref{lem-r_cut}, Distal cutting lemma) Let $\mathcal{M}$ be a first-order structure. Let $\varphi(x;y)$ be a formula admitting a distal cell decomposition $\CT$ (given by a finite set of formulas $\Psi(x; \bar{y})$ --- see Definition \ref{def: def cell decomp}) with 
$|\CT(S)| = O(|S|^d)$  (i.e. for some constant $C\in \RR$, for any \textbf{non-empty} 
finite $S\subseteq M^{|y|}$ we have $|\CT(S)| \leq C|S|^d$). 

Then for any finite $H \subseteq M^{|y|}$ of size $n$ and any real $r$ satisfying $1 < r < n$,  there are subsets $X_1, \ldots, X_t$ of $M^{|x|}$  covering $M^{|x|}$ with
$$
t \leq Cr^d
$$
for some constant $C = C(\varphi)$ (and independent of $H$, $r$ and $n$), and with each $X_i$ crossed by at most $n/r$ of the formulas $\{\varphi(x;a): a \in H\}.$ 

Moreover, each $X_i$ is the intersection of at most two sets $\Psi$-definable over $H$ (see Definition \ref{def: definable, crossing}).

\end{thm*}

While every formula in a distal structure admits a distal cell decomposition (see Fact \ref{fac: char of distality}), establishing optimal bounds in dimension higher than $1$ is non-trivial. In our second theorem, we demonstrate that formulas in o-minimal structures admit distal cell decompositions of optimal size ``on the plane''.

\begin{thm*}(Theorem \ref{thm:cell-dec})
Let $\CM$ be an o-minimal expansion of a real closed field. 
For any  formula $\varphi(x;y)$ with $|x|=2$ there is a distal cell decomposition $\CT$ with 
$|\CT(S)| = O(|S|^2)$.
\end{thm*}

In our proof, we show that a version of the vertical cell decomposition can be generalized to arbitrary o-minimal theories. This gives an optimal bound for subsets of $M^2$, but determining the exact bounds for distal cell decompositions in higher dimensions remains open, even in the semialgebraic case.

Finally, in Section \ref{sec: Zarank} we apply these two theorems to generalize the results in \cite{zaran} on the semialgebraic Zarankiewicz problem to arbitrary o-minimal structures, in the planar case (our result is more general and applies to arbitrary definable families admitting a quadratic distal cell decomposition, see Section \ref{sec: Zarank} for the precise statements).

\begin{thm*} (Theorem \ref{thm: everything o-min})
Let $\CM$ be an o-minimal expansion of a real closed field and let $E(x,y) \subseteq M^2 \times M^d$ be a definable relation, given by an instance of some formula $\theta(x,y;z) \in \mathcal{L}$ using some parameters from $M^{|z|}$. 

\begin{enumerate}
\item For every $k \in \mathbb{N}$ there is a constant $\alpha = \alpha(\theta,k) \in \mathbb{R}$ such that for any finite $P \subseteq M^2, Q \subseteq M^d$, $|P|=m, |Q| = n$, if $E \cap (P \times Q)$ does not contain a copy of $K_{k,k}$ (the complete bipartite graph with two parts of size $k$), then we have 
	$$ |E(P,Q)| \leq \alpha \left( m^{\frac{d}{2d-1}} n^{\frac{2d-2}{2d-1}} + m + n \right).$$

	\item There is some $k' \in \mathbb{N}$ and formulas $\varphi(x,v), \psi(y,w)$, all depending only on $\theta$, such that if $E$ contains a copy of $K_{k',k'}$, then there are some parameters $b \in M^v, c \in M^w$ such that both $\varphi(M,b)$ and $ \psi(M,c)$ are infinite and $\varphi(M,b) \times \psi(M,c) \subseteq E$.
\end{enumerate}
\end{thm*}

Combining the two parts, it follows that either $E$ contains a product of two infinite definable sets, or the upper bound on the number of edges in part (1) holds for all finite sets $P,Q$ with some fixed constant $\alpha =\alpha (\theta)$.

The special case with $d = 2 $ can be naturally viewed as a generalization of the classical Szemer\'edi-Trotter theorem for o-minimal structures.
\begin{cor}\label{cor: Sz-Trot}
	Let $\CM$ be an o-minimal expansion of a real closed field. Then for every $\theta$-definable relation $E(x,y) \subseteq M^2 \times M^2$ there is a constant $\alpha \in \mathbb{R}$ and some formulas $\varphi(x,v), \psi(y,w)$,  depending only on $\theta$, such that exactly one of the following occurs:

\begin{enumerate}
	\item For any finite $P \subseteq M^2, Q \subseteq M^2$, $|P|=m, |Q| = n$ we have $$ |E(P,Q)| \leq \alpha \left( m^{\frac{2}{3}} n^{\frac{2}{3}} + m + n \right),$$
	\item there are some parameters $b \in M^v, c \in M^w$ such that both $\varphi(M,b)$ and $ \psi(M,c)$ are infinite and $\varphi(M,b) \times \psi(M,c) \subseteq E$.
\end{enumerate}

\begin{rem}
While this paper was in preparation, we have learned that Basu and Raz \cite{basu2016minimal} have obtained a special case of Corollary \ref{cor: Sz-Trot} using different methods.
\end{rem}

\end{cor}

\subsection*{Acknowledgements} 
We thank Shlomo Eshel and the anonymous referee for pointing out some inaccuracies and suggestions on improving the paper.
Chernikov was supported by the NSF Research Grant DMS-1600796, by the NSF CAREER grant DMS-1651321 and by an Alfred P. Sloan Fellowship.
Galvin was supported by the Simons Foundation.
Starchenko was supported by the NSF Research Grant DMS-1500671.

\section{Preliminaries and the distal cell decomposition}\label{sec: prelims and distality}

Throughout this section we fix  a first-order structure $\CM$  in a language $\CL$. At this point we don't make any additional assumptions on $\CM$, e.g. we may work in ``set theory'',
i.e. in a structure where every subset is definable.  
We introduce some basic notation and terminology. Given a tuple of variables $x$, we let $|x|$ denote its length. For each $n \in \mathbb{N}$, $M^n$ denotes the corresponding cartesian power of $M$, the underlying set of $\CM$.
For a fixed formula $\varphi(x;y) \in \CL$ with two groups of variables $x$ and $y$, given $b \in M^{|y|}$ we write $\varphi(M;b)$ to denote the set $\{ a \in M^{|x|} : \CM \models \varphi(a;b)  \}$. Hence the formula $\varphi(x;y)$ can be naturally associated with the \emph{definable family} of sets $\{ \varphi(M;b) : b \in M^{|y|}\}$. E.g., if $\CM$ is the field of reals, all sets in such a family for a fixed $\varphi(x;y)$ are semialgebraic of description complexity bounded by some $d = d(\varphi)$ and conversely, the  family of all semialgebraic sets of description complexity bounded by some fixed $d$ can be obtained in this way for an appropriate choice of the formula $\varphi(x;y)$. We refer to  \cite{distal} for a more detailed introduction and examples of the relevant model-theoretic terminology.

\begin{defn}
For sets $A,X \subseteq M^{d}$ 
we say that $A$ \emph{crosses} $X$ if both 
$X\cap A$ and $X\cap \neg A$ are nonempty.
\end{defn}

We extend the above definition to a set of formulas. 

\begin{defn}\label{def: definable, crossing} Let $\Phi(x;y)$ be a set of $\CL$-formulas of the form $\varphi(x;y)$ and $S\subseteq
  M^{|y|}$.
 \begin{enumerate}
  \item We say that a subset $A\subseteq M^{|x|}$ is \emph{$\Phi(x;S)$-definable}
    if $A=\varphi(M;s)$ for some $\varphi(x;y)\in \Phi$ and $s\in S$. 
  \item For a set $\Delta\subseteq M^{|x|}$ we say that \emph{$\Phi(x;S)$ 
    crosses $\Delta$} if some $\Phi(x;S)$-definable set crosses $\Delta$. In
  other words $\Phi(x;S)$ does not cross $\Delta$ if for any $\varphi(x;y)
  \in \Phi(x;y)$ and $s\in S$ the formula  $\varphi(x;s)$ has a constant truth
  value on $\Delta$. 
  \end{enumerate}
\end{defn}

We define a very general combinatorial notion of an abstract cell decomposition for formulas (equivalently, for definable families of sets).

\begin{defn}\label{def: cell decomp}Let $\Phi(x;y)$ be a finite set of formulas. 
  \begin{enumerate}
  \item Given a finite set $S \subseteq M^{|y|}$, a finite family $\CF$ of subsets of $M^{|x|}$ is
    called \emph{ an abstract  cell decomposition for $\Phi(x;S)$} if $M^{|x|}=\cup
      \CF$  and every $\Delta\in \CF$ is not crossed by $\Phi(x;S)$. 
    \item  \emph{An abstract cell decomposition for  $\Phi(x;y)$} is  an assignment $\CT$ that to each finite
set $S\subseteq M^{|y|}$ assigns an abstract  cell decomposition  $\CT(S)$ for $\Phi(x;S)$.
  \end{enumerate}
\end{defn}

\begin{rem}
  In the above definition, the term ``cell decomposition'' is understood in a very weak sense. Firstly, the ``cells'' in $\CT(S)$ are not required to have any ``geometric'' properties, and  secondly, we don't require the family
  $\CT(S)$ to partition $M^{|x|}$, but only ask for it to be a covering.
 \end{rem}

Every $\Phi(x;y)$ admits an obvious abstract cell decomposition, with $\CT(S)$ consisting of the atoms in the Boolean algebra generated by the $\Phi(x;S)$-definable sets. In general, defining these cells would require longer and longer formulas when $S$ grows, and the aim of the following definitions is to avoid this possibility.

\begin{defn}
Let $\Phi(x;y)$ be a finite set of formulas and $\CT$ an abstract  cell
decomposition for $\Phi(x;y)$. 

We say that $\CT$ is \emph{weakly
  definable} if there is a finite set of formulas
$\Psi(x; \bar y)=\Psi(x;y_1,\dotsc,y_k)$ with
$|y_1|=\ddotsb=|y_k|=|y|$ such that for any finite $S\subseteq M^{|y|}$, 
every $\Delta\in \CT(S)$ is  $\Psi(x;S^k)$-definable  (i.e., $\Delta=\psi(M; s_1,\dotsc,s_k)$ for some $s_1,\dotsc,s_k \in S$ and $\psi\in \Psi$).  
In this case we also say that $\Psi(x,\bar y)$ weakly defines $\CT$.

 \end{defn}

 \begin{rem}\label{rem:sergei-new}
If  $\CT$ is  an abstract  cell decomposition for $\Phi(x;y)$ that is  weakly
defined by $\Psi(x;\bar y)$   then $\Psi(x;\bar y)$ does not determine
$\CT$ uniquely.  However there is a maximal abstract cell decomposition
$\CT^{\textrm{max}}$  weakly defined by  $\Psi(x;\bar y)$, where 
 $\CT^{\textrm{max}}(S)$ consists of \emph{all} $\Psi(x;S^k)$-definable sets $\Delta$ 
such that $\Phi(x;S)$ does not cross $\Delta$. 
 \end{rem}

For combinatorial applications discussed in this paper it is desirable  to have a cell decomposition with as few sets as
possible, and also  to have control over the sets appearing in
$\CT(S)$ in a definable way.

\begin{defn}\label{def: def cell decomp} Let $\Phi(x;y)$ be a finite set of formulas. 
We say that an abstract  cell decomposition $\CT$  
for $\Phi$ is \emph{definable} if it is weakly defined by some $\Psi(x;y_1,\dotsc,y_k)$ and 
if for every finite $S\subseteq M^{|y|}$  and each $\Psi(x; S^k)$-definable $\Delta \subseteq M^{|x|}$
there is  a set $\CI(\Delta)\subseteq M^{|y|}$, uniformly definable in $\Delta$, 
such that 
\begin{equation}
  \label{eq:1}
\CT(S)=\{ \Delta\in \Psi(S) \colon \CI(\Delta)\cap S 
=\emptyset\}. 
 \end{equation}
By the uniform definability of $\CI(\Delta)$ we mean  that  for every
$\psi(x;\bar y )\in \Psi(x;\bar y)$ there is a formula
$\theta_\psi(y;\bar y)$ such that for any $s_1,\dotsc,s_k\in
M^{|y|}$ if $\Delta=\psi(M;s_1,\dotsc,s_k)$ then 
$\CI(\Delta)=\theta_\psi(M;s_1,\dotsc,s_k)$. 
\end{defn}

For example, $\CT^{\textrm{max}}$ from Remark~\ref{rem:sergei-new} is definable with $\CI(\Delta)=\{ s\in M^{|y|} \colon   
 \Phi(x;s) \text{ crosses }\Delta \}$.

\begin{rem}\label{rem: I Delta contains crossing params}
It follows from Definition \ref{def: def cell decomp} that for every $\Psi(x;M)$-definable set $\Delta \subseteq M^{|x|}$, the set of all $s \in M^{|y|}$ such that $\Phi(x;s)$ crosses $\Delta$ is contained in $\CI(\Delta)$ (strict containment is possible, however).

Indeed, assume that $s \in M^{|y|}$ and $\varphi(x;y) \in \Phi$ are such that $\varphi(x;s)$ crosses $\Delta$. By Definition \ref{def: cell decomp}(1), necessarily $\Delta \notin \CT(\{ s \})$. But then $\CI(\Delta) \cap \{s \} \neq \emptyset$ by  (\ref{eq:1}), hence $s \in \CI(\Delta)$.
\end{rem}

\begin{rm}

\end{rm}

As it was shown  in \cite{distal}, such combinatorial definable cell decompositions have a close connection to the model-theoretic notion of distality. Distal structures were introduced in \cite{DistalPierre} for purely model theoretic purposes (we don't give the original definition here). The following fact was pointed out in \cite{distal} and can be used as the definition of a distal structure in this paper. 
\begin{fact} \label{fac: char of distality} The following are equivalent for a first-order structure $\CM$.
\begin{enumerate}
	\item $\CM$ is distal,
	\item for every formula $\varphi(x;y)$ there is a weakly definable cell
  decomposition for $\{ \varphi(x;y) \}$,
  \item for every formula $\varphi(x;y)$ there is a definable cell decomposition for $\{ \varphi(x;y) \}$.
	\end{enumerate}
\end{fact}

 Indeed,  equivalence of the original definition of distality and existence of weakly definable cell
  decompositions is given by \cite[Theorem 21]{chernikov2015externally}; and if $\CT$ is a weakly definable cell decomposition for $\varphi(x;y)$, then $\CT^{\textrm{max}}$  from Remark~\ref{rem:sergei-new} is definable.
 
 Examples of distal structures include: 
 \begin{enumerate}
 	\item o-minimal structures;
 	\item Presburger arithmetic $(\mathbb{Z}, +, 0, <)$;
 	\item the field of $p$-adics $\mathbb{Q}_p$;
 \end{enumerate}
  (we refer to the introduction of \cite{distal} for a more detailed discussion).

There are several contexts in model theory relevant for the topics of this paper where certain notions of cell decomposition play a prominent role (e.g. o-minimal cell decomposition, $p$-adic cell decomposition, etc.). These cell decompositions tend to carry more geometric information, while the one discussed here captures combinatorial complexity. To distinguish from those cases, and in view of Fact \ref{fac: char of distality}, we will from now on refer to a definable cell decomposition $\CT$ for a finite set of formulas $\Phi(x;y)$ as in Definition \ref{def: def cell decomp} as a \emph{distal cell decomposition} for $\Phi(x;y)$. Hence, \textbf{a structure $\CM$ is distal if and only if every formula admits a distal cell decomposition}.
  
   Distality of the examples listed above had been established by different (sometimes infinitary) methods and the question of obtaining the exact bounds on the size of the corresponding distal cell decompositions hasn't been considered. While it is easy to verify in the examples listed above that all formulas $\varphi(x,y)$ with $|x|=1$ admit a cell decomposition $\CT$ with the best possible bound $|\CT(S)| = O(|S|)$, already the case of formulas with $|x|=2$ becomes more challenging (and grows in complexity with $|x|$). In Section \ref{sec: omin-cell-decomp} we establish that in an o-minimal expansion of a field, all formulas with $|x|=2$ admit a distal cell decomposition $\CT$ with the optimal bound $|\CT(S)| = O(|S|^2)$ (the case $|x|\geq 3$ remains open, even in the semialgebraic case).   
   
%

\section{Distal cutting lemma}
In this section we show how a bound on the size of a distal cell decomposition for a given definable family can be used to obtain a definable cutting lemma with the corresponding bound 
on its size.

Our proof generalizes (and closely follows) the axiomatic treatment of the Clarkson-Shor random sampling technique in \cite[Section 6.5]{mbook}. 

\begin{defn}
($\frac{1}{r}$-cutting) Let $\mathcal{F}$ be a finite family of subsets of a set $X$ with $|\mathcal{F}| = n$. Given a real $r \geq 1$, we say that a family $\mathcal{C}$ of subsets of $X$ is an \emph{$\frac{1}{r}$-cutting for $\mathcal{F}$} if the sets in $\mathcal{C}$ form a covering of $X$ and each set in $\mathcal{C}$ is crossed by at most $\frac{n}{r}$ sets in $\mathcal{F}$.

\end{defn}

Throughout this section we fix  a first-order structure $\CM$  in a language $\CL$.

\begin{thm} \label{lem-r_cut}
(Distal cutting lemma) Let $\varphi(x;y) \in \mathcal{L}$ be a formula admitting a distal cell decomposition $\CT$ (weakly defined by a  finite set of formulas $\Psi(x; y_1, \ldots, y_s)$ --- see Definition \ref{def: def cell decomp}) with 
$|\CT(S)| = O(|S|^d)$. 

Then for any finite $H \subseteq M^{|y|}$ of size $n$ and any real $r$ satisfying $1 < r < n$,  the family $\{\varphi(M;a): a \in H\}$ of subsets of $M^{|x|}$ admits a $\frac{1}{r}$-cutting $X_1, \ldots, X_t$ with
$$
t \leq Cr^d
$$
for some constant $C = C(\varphi)$ (and independent of $H$, $r$ and $n$).

Moreover, each of the $X_i$'s is an intersection of at most two $\Psi(x;H^s)$-definable sets.
\end{thm}

\begin{rem}\label{rem: triv cut r geq n}
We note that Theorem \ref{lem-r_cut} is trivially true for $r=1$ (with $t=1$ and $X_1 = X$), and for $r \geq n$ since the distal cell decomposition itself will give a desirable partition in that case.	
\end{rem}

In the rest of this section we present a proof of Theorem~\ref{lem-r_cut}.

\medskip

We fix  $\CT, \Psi$ and $H$ as in the assumption of the theorem.

By  Definition \ref{def: def cell decomp}, for each finite $S \subseteq M^{|y|}$, we have a finite collection ${\mathcal T}(S)$ of subsets of $M^{|x|}$ that covers $M^{|x|}$ and satisfies the following conditions.

\begin{enumerate}[labelindent=10pt]
	\item[\textbf{(C1)}] Let 
$$
{\rm Reg} := \{\Delta : \Delta \in {\mathcal T}(S)~\mbox{for some $S \subseteq H$}\}.
$$  
Then every set in {\rm Reg} is definable by an instance of a formula from $\Psi$ with parameters in $H$.
	\item[\textbf{(C2)}] For every $S \subseteq H$ we have
$$
|{\mathcal T}(S)| \leq C'\left(|S|^d+1\right)
$$
for some constant $C'$ depending only on $\varphi$. (The hypothesis of the theorem ensures that for \emph{non-empty} $S$ we have $|{\mathcal T}(S)| \leq C|S|^d$ for some constant $C=C(\varphi)$. We add ``$+1$'' here to take into account the case $S = \emptyset$.)



\item[\textbf{(C3)}] Let $\Delta \in {\rm Reg}$.   We associate to it a collection ${\mathcal D}(\Delta)$ of subsets of $H$, called the {\em defining sets} of $\Delta$, via 
$$
{\mathcal D}(\Delta) := \{S \subseteq H: |S| \leq s,~\Delta \in {\mathcal T}(S)\}.
$$ 
(Here $s$ is a fixed constant corresponding to the number of parameters in $\Psi(x;y_1, \ldots, y_s)$ given by the distal cell decomposition and depending only on $\varphi$).

Given $\CI$ as in Definition \ref{def: def cell decomp}, we define $\CI_H(\Delta) := \CI(\Delta) \cap H$.  Notice that $\CI_H(\Delta)$ contains all of the $a \in H$ such that $\varphi(x;a)$ crosses $\Delta$ (by Remark \ref{rem: I Delta contains crossing params}). 

We have:
$$
\Delta \in {\mathcal T}(S) \iff {\mathcal I}_H(\Delta) \cap S = \emptyset~\mbox{and there is}~S_0 \in {\mathcal D}(\Delta)~\mbox{with}~S_0 \subseteq S.
$$

\end{enumerate}

\begin{rem} It  follows from the proof that the distal cutting lemma (Theorem \ref{lem-r_cut}) holds for any abstract cell decomposition satisfying the conditions (C1)--(C3) with an appropriately chosen relation $\CI(\Delta)$. 
\end{rem}

\medskip

Before proceeding to the proof of the distal cutting lemma (Theorem \ref{lem-r_cut}) we isolate two key tools. The first is a tail bound on the probability that a cell $\Delta \in {\mathcal T}(S)$ is crossed by many formulas, where $S$ is a randomly chosen subset of $H$.

For $S \subseteq H$ and $t\geq 0$ let ${\mathcal T}(S)_{\geq t}$ denote the set of $\Delta \in {\mathcal T}(S)$ with $|{\mathcal I}_H(\Delta)| \geq tn/r$. Recall that for $0 \leq p \leq 1$ we say that $S \subseteq H$ is selected by {\em independent Bernoulli trials with success probability $p$} if $S$ is selected according to the distribution $\mu$ (supported on the power set of $H$) given by 
$$
\mu(S') = p^{|S'|}(1-p)^{|H|-|S'|}
$$
for each $S' \subseteq H$; observe that this is essentially the process of flipping a biased coin (biased to show heads with probability $p$) $|H|$ times independently, and for $1 \leq i \leq |H|$ putting the $i$th element of $H$ in $S$ if and only if the $i$th flip comes up heads. 

\begin{lem} \label{lem-tail}
(Tail Bound Lemma) Let $\varphi(x;y)$ be a formula as in Theorem \ref{lem-r_cut}. Let $H \subseteq M^{|y|}$ be a finite set of size $n$. Fix $\varepsilon > 0$ and let $r$ be a parameter satisfying $1 \leq r \leq (1-\varepsilon)n$. Let $S \subseteq H$ be selected by independent Bernoulli trials with success probability $r/n$, and let $t \geq 0$ be given. Then there is a constant $C=C(\varepsilon)$ such that
$$
{\bf E}_{\mu}\left(\left|{\mathcal T}(S)_{\geq t}\right|\right) \leq C2^{-t}r^d.
$$
\end{lem} 

We use this to derive the second main tool, a cutting lemma that is weaker than Theorem \ref{lem-r_cut}. Here and everywhere else, all logarithms are base $2$.
\begin{lem} \label{lem-subopt_cut}
(Suboptimal Cutting Lemma) Let $\varphi(x;y)$ be a formula as in Theorem \ref{lem-r_cut}. Let $H \subseteq M^{|y|}$ be a finite set of size $n$. Let $r$ be a parameter satisfying $1 < r < n$. There is $S \subseteq H$ with
$$
|{\mathcal T}(S)| \leq Kr^d\log^d (r+1)
$$
for some constant $K$ independent of $H$, $r$ and $n$, and with each $X \in {\mathcal T}(S)$ crossed by at most $n/r$ of the formulas $\{\varphi(x;a): a \in H\}.$  
\end{lem}

\medskip

\begin{proof}[Proof (assuming Lemma~\ref{lem-tail})] 

Let $A$ be such that $3\times 2^{2d}CA^d = 2^A$, where $C$ is the constant appearing in Lemma \ref{lem-tail}. Increasing $C$ if necessary, we may assume that $A \geq 1$. We treat separately the cases $2Ar \log(r+1) \leq n$ and $2Ar \log(r+1) \geq n$. If $2Ar \log(r+1) \geq n$ then we may take $S=H$, since $\CT(H)$ has size  $C'(n^d +1) \leq C'((2A)^dr^d \log^d(r+1) + 1) \leq Kr^d\log^d (r+1)$ for suitably large $K$ (note that by \textbf{(C3)} no instance of $\varphi(x;y)$ over $H$ can cross any of the sets in ${\mathcal T}(H)$).

Suppose now that $2Ar \log(r+1) \leq n$. Set $r' = Ar\log (r+1)$, and note that $r' \geq 1$ as $A \geq 1$, $r > 1$ and $\log$ is base $2$. Applying Lemma \ref{lem-tail} with $r'$ taking the role of $r$ (valid since $r' < n/2$) and with $t=0$ we obtain that if $S \subseteq H$ is selected by independent Bernoulli trials with success probability $r'/n$ (with associated distribution $\mu'$) then
$$
{\bf E}_{\mu'}\left(\left|{\mathcal T}(S)\right|\right) \leq CA^dr^d \log^d(r+1).
$$
Applying Lemma \ref{lem-tail} again with $t=A \log(r+1)$ we get
$$
{\bf E}_{\mu'}\left(\left|{\mathcal T}(S)_{\geq A \log(r+1)}\right|\right) \leq \frac{CA^d r^d \log^d(r+1)}{(r+1)^A} \leq \frac{CA^d}{(r+1)^{A-2d}} \leq 1/3,
$$
the second inequality using $r\log(r+1) \leq (r+1)^2$ and the third using our choice of $A$ and the fact that $r \geq 1$.  By linearity of expectation
$$
{\bf E}_{\mu'}\left(\frac{\left|{\mathcal T}(S)\right|}{3CA^dr^d \log^d(r+1)} + \left|{\mathcal T}(S)_{\geq A \log(r+1)}\right| \right) \leq 2/3,
$$
so there exists an $S \subseteq H$ such that 
$$
\left|{\mathcal T}(S)\right| \leq 3CA^dr^d \log^d(r+1)
$$
and ${\mathcal T}(S)_{\geq A \log(r+1)} = \emptyset$. This last condition implies that each $\Delta \in {\mathcal T}(S)$ is crossed by at most
$(A \log(r+1) n)/r' = n/r$
formulas. 

%
\end{proof}

\medskip

We use Lemmas \ref{lem-subopt_cut} and \ref{lem-tail} to derive Theorem \ref{lem-r_cut}, before turning to the proof of Lemma \ref{lem-tail}.

\medskip

\begin{proof}[Proof of Theorem \ref{lem-r_cut}]
Just as in the proof of Lemma \ref{lem-subopt_cut} we begin by observing that the family ${\mathcal T}(H)$ itself satisfies the conclusion of Theorem \ref{lem-r_cut} for all $r$, with size at most $C'(n^d +1)$. This allows us to assume, say, $r \leq n/2$ and use Lemma \ref{lem-tail}.

Let $S \subseteq H$ be selected by independent Bernoulli trials with success probability $r/n$.

For $\Delta \in {\mathcal T}(S)$ define $t_\Delta$ by $|{\mathcal I}_H(\Delta)| = t_\Delta n/r$. Note that if $t_\Delta \leq 1$ then the number of $a$ in $H$ such that $\varphi(x,a)$ crosses $\Delta$ is no more than $n/r$.

For $\Delta \in {\mathcal T}(S)$ with $t_\Delta > 1$, consider the set ${\mathcal I}_H(\Delta)$. It contains all $a \in H$ for which $\varphi(x,a)$ crosses $\Delta$.  By Lemma \ref{lem-subopt_cut} there is $S' \subseteq {\mathcal I}_H(\Delta)$ with ${\mathcal T}(S')$ having size at most $O(t^d_\Delta \log^d (t_\Delta+1))$ with the property that for every $\Delta' \in {\mathcal T}(S')$, the number of $a \in {\mathcal I}_H(\Delta)$ such that $\varphi(x,a)$ crosses $\Delta'$ is at most
$$
\frac{|{\mathcal I}_H(\Delta)|}{t_\Delta} = \frac{n}{r}.
$$       
In particular that means that for every $\Delta' \in {\mathcal T}(S')$ the number of $a \in H$ such that $\varphi(x,a)$ crosses $\Delta' \cap \Delta$ is at most $n/r$.

It follows that the family of subsets of $M^{|x|}$ consisting of those $\Delta \in {\mathcal T}(S)$ for which $t_\Delta \leq 1$, together with all sets of the form $\Delta' \cap \Delta$ where $\Delta \in {\mathcal T}(S)$ has $t_\Delta > 1$ and $\Delta' \in {\mathcal T}(S')$ (with $S'$ constructed from $S$ via Lemma \ref{lem-subopt_cut}, as described above), forms a cover of $M^{|x|}$ with size at most
\begin{equation} \label{eq-decomp_bound}
\sum_{\Delta \in {\mathcal T}(S)} \left({\bf 1}_{\{t_\Delta \leq 1\}} + Ct_\Delta^d\log^d(t_\Delta+1) {\bf 1}_{\{t_\Delta > 1\}}\right).
\end{equation}
We now upper bound the expectation (with respect to $\mu$) of this quantity. By linearity the expectation is at most
\begin{equation} \label{eq-size}
{\bf E}_{\mu}\left(\left|{\mathcal T}(S)\right|\right) + C\sum_{i \geq 0} {\bf E}_{\mu}\left(\sum_{\Delta \in {\mathcal T}(S) \colon 2^i \leq t_\Delta < 2^{i+1}} t_\Delta^{2d}\right)
\end{equation}
(using $\log (t_\Delta+1) \leq t_\Delta$ for $t_\Delta \geq 1$). 

We bound the first term in (\ref{eq-size}) by an application of Lemma \ref{lem-tail} with $t=0$. This gives
$$
{\bf E}_{\mu}\left(\left|{\mathcal T}(S)\right|\right) \leq O(r^d).
$$
For the second term in (\ref{eq-size}) we have
\begin{multline*}
\sum_{i \geq 0} {\bf E}_\mu \left(\sum_{\Delta \in {\mathcal T}(S) \colon 2^i \leq t_\Delta < 2^{i+1}} t_\Delta^{2d}\right) \leq  \sum_{i \geq 0} 2^{2d(i+1)} {\bf E}_{\mu}\left(\left|{\mathcal T}(S)_{\geq 2^i}\right|\right) \\
 \leq  C' \sum_{i \geq 0} 2^{2d(i+1)} 2^{-2^i} r^d = O(r^d),
\end{multline*}
with the last inequality and the constant $C'$ given by an application of Lemma \ref{lem-tail}. 

We conclude that the expectation of the quantity in (\ref{eq-decomp_bound}) is $O(r^d)$, so there is at least one choice of $S \subseteq H$ for which (\ref{eq-decomp_bound}) is at most $O(r^d)$, proving Theorem \ref{lem-r_cut} (the definability clause follows by \textbf{(C1)} as every set in the constructed covering is an intersection of at most two sets from ${\rm Reg}$).  
\end{proof}

\medskip

Before proving Lemma \ref{lem-tail} we isolate a useful set-systems lemma.
\begin{lem} \label{lem-set_system_corr}
Let $\Omega$ be a set of size $m$, and let $\{D_1, \ldots, D_q\}$ be a collection of subsets of $\Omega$ with $|D_i| \leq u$ for all $i$, $1 \leq i \leq q$, for some $u$. Let
$$
{\mathcal F} = \{X \subseteq \Omega:D_i \subseteq X~\mbox{for some $i$, $1 \leq i \leq q$}\}
$$
be the ``up-set'' generated by the $D_i$'s. Let $\tilde{p}$ and $p$ satisfy $0 < \tilde{p} \leq p \leq 1$. We have
\begin{equation} \label{eq-correlation_inq}
\frac{\sum_{X \in {\mathcal F}} \tilde{p}^{|X|}(1-\tilde{p})^{m-|X|}}{\sum_{X \in {\mathcal F}} p^{|X|}(1-p)^{m-|X|}} \geq \left(\frac{\tilde{p}}{p}\right)^u.
\end{equation}
\end{lem}

\begin{proof}
With each $X \in {\mathcal F}$ associate (arbitrarily) a set $D_X$ satisfying $D_X \subseteq X$ and $D_X \in \{D_1, \ldots, D_q\}$ (such a set exists by the definition of ${\mathcal F}$).

Let $A \subseteq \Omega$ be selected by independent Bernoulli trials with success probability $p$, and, independently, let $B \subseteq \Omega$ be selected by independent Bernoulli trials with success probability $\tilde{p}/p$. Observe that
\begin{equation} \label{eq-first_random_selection}
\Pr(A \in {\mathcal F}) = \sum_{X \in {\mathcal F}} p^{|X|}(1-p)^{m-|X|}
\end{equation} 
and
\begin{equation} \label{eq-combination_random_selection}
\Pr(A \cap B \in {\mathcal F}) = \sum_{X \in {\mathcal F}} \tilde{p}^{|X|}(1-\tilde{p})^{m-|X|},
\end{equation}
with (\ref{eq-combination_random_selection}) holding by independence and because for each $\omega \in \Omega$, $\Pr(\omega \in A \cap B)=\Pr(\omega \in A)\Pr(\omega \in B)$. 
  
Now consider the two events
$$
E_1 = \{A \in {\mathcal F}~\mbox{and}~D_A \subseteq B\}
$$  
and
$$
E_2 = \{A \cap B \in {\mathcal F}\}.
$$
If $A \in {\mathcal F}$ and $D_A \subseteq B$ then $D_A \subseteq A \cap B$, so that $A \cap B \in {\mathcal F}$. It follows that $E_1 \subseteq E_2$ and 
\begin{equation} \label{eq-containment}
\Pr(E_1) \leq \Pr(E_2).
\end{equation}
Using independence we have
\begin{multline*}
\Pr(E_1) = \sum_{X \in {\mathcal F}} \Pr(A=X~\mbox{and}~D_X \subseteq B) = \sum_{X \in {\mathcal F}} \Pr(A=X)\left(\frac{\tilde{p}}{p}\right)^{|D(X)|}  \\
 \geq \Pr(A \in {\mathcal F})\left(\frac{\tilde{p}}{p}\right)^u.
\end{multline*}
Combining with (\ref{eq-first_random_selection}), (\ref{eq-combination_random_selection}) and (\ref{eq-containment}) we get (\ref{eq-correlation_inq}).
\end{proof}

\medskip

We are now ready to prove Lemma \ref{lem-tail}. We follow Matou\v{s}ek's approach in \cite[Section 6.5]{mbook}, but add an additional argument. 
\begin{proof}[Proof of Lemma \ref{lem-tail}] 
We start by establishing
\begin{equation} \label{claim1}
{\bf E}_\mu\left(\left|{\mathcal T}(S)\right|\right) = O(r^d),
\end{equation}
which gives Lemma \ref{lem-tail} for $t \leq 1$.
To see (\ref{claim1}) note that \textbf{(C2)} yields ${\bf E}_\mu\left(\left|{\mathcal T}(S)\right|\right) \leq C{\bf E}_\mu\left(|S|^d\right)+1$. Now $|S|=X_1+\ldots + X_n$ where the $X_i's$ are independent Bernoulli random variables each with parameter $p=r/n$. We claim that for all $d \geq 1$ we have
\begin{equation} \label{eq-bin-moment-bound}
{\bf E}_\mu(|S|^d) \leq (r+d)^d.
\end{equation}
(from which (\ref{claim1}) immediately follows; note that we can drop the $+1$ since $r \geq 1$).

To see (\ref{eq-bin-moment-bound}), note first that by linearity we have 
$$
{\bf E}_\mu(|S|^d) = \sum_{(i_1,i_2,\ldots,i_d) \in \{1, \ldots, n\}^d} {\bf E}(X_{i_1}X_{i_2}\cdots X_{i_d}).
$$
Let $a_k$ be the number of tuples $(i_1,i_2,\ldots,i_d) \in \{1, \ldots, n\}^d$ such that $|\{i_1,i_2,\ldots,i_d\}|=d-k$. By independence of the $X_i$, and the fact that $X_i^\ell$ has the same distribution as $X_i$ for any integer $\ell \geq 1$ we have
\begin{equation} \label{eq-bin-moment-bound_int}
{\bf E}_\mu(|S|^d) = \sum_{k=0}^d a_k p^{d-k}.
\end{equation}
We claim that 
\begin{equation} \label{eq-bin-moment-bound_int2}
a_k \leq \binom{d}{k}d^kn^{d-k}.
\end{equation}
Inserting into (\ref{eq-bin-moment-bound_int}) and using the binomial theorem together with $np=r$, this gives (\ref{eq-bin-moment-bound}).

To see (\ref{eq-bin-moment-bound_int2}) note that we overcount $a_k$ by first specifying $d-k$ indices from $\{1,\ldots,d\}$ on which the $i_j$'s are all different from each other ($\binom{d}{d-k}=\binom{d}{k}$ choices), then choosing values for these $i_j$'s ($n(n-1)\cdots(n-(d-k)+1) \leq n^{d-k}$ choices), and finally choosing values for the remaining indices ($(d-k)^k \leq d^k$ choices, since these indices are all constrained to lie among the $d-k$ distinct indices chosen initially). It follows that $a_k \leq \binom{d}{k} n^{d-k} d^k$, as claimed.

(We note that in the case $d=2$ things are considerably easier: we have
$$
|S|^2 = \sum_{i=1}^n X_i^2 + 2\sum_{1 \leq i < j \leq n} X_iX_j
$$
so
\begin{align*}
{\bf E}_\mu(|S|^2) & =  \sum_{i=1}^n {\bf E}(X_i^2) + 2\sum_{1 \leq i < j \leq n} {\bf E}(X_iX_j) \\
& =  np + n(n-1)p^2 \\
& \leq  np + n^2p^2 = r^2+r.)
\end{align*}

We assume from now on that $t\geq 1$. For $\Delta\in {\rm Reg}$ denote by $p(\Delta)$ the probability that
$\Delta$ appears in ${\mathcal T}(S)$,  i.e.
$$
p(\Delta)=  \mu(\{ S\subseteq H \colon \Delta\in {\mathcal T}(S)\})=\sum_{\Delta\in {\mathcal T}(S)} \mu(S).
$$
Let ${\rm Reg}_{\geq t}=\{ \Delta\in  {\rm Reg} \colon |{\mathcal I}_H(\Delta)|\geq tn/r\}$. By linearity of expectation we have
\begin{equation} \label{eq:2}
{\bf E}_\mu\left(\left|{\mathcal T}(S)_{\geq t}\right|\right)= \sum_{\Delta\in {\rm Reg}_{\geq t}} p(\Delta).
\end{equation}

Now set $\tilde{p}=p/t$ and let $\tilde{\mu}$ be the distribution associated with selection from $H$ by independent Bernoulli trials with success probability $\tilde{p}$. By (\ref{claim1}) we have 
\begin{equation} \label{eq:3}
{\bf E}_{\tilde{\mu}}\left(\left|{\mathcal T}(S)\right|\right)=O(r^d/t^d).
\end{equation}
Also, as in \eqref{eq:2} we have   
\begin{multline}\label{eq:4}
{\bf E}_{\tilde{\mu}}\left(\left|{\mathcal T}(S)\right|\right)  =  \sum_{\Delta\in {\rm Reg}}
\tilde{p}(\Delta) 
 \geq  \sum_{\Delta\in {\rm Reg}_{\geq t}} \tilde{p}(\Delta) \\
 =  \sum_{\Delta\in {\rm Reg}_{\geq t}} p(\Delta)\frac{\tilde{p}(\Delta)}{p(\Delta)}  
 \geq   \min\left\{ \frac{\tilde{p}(\Delta)}{p(\Delta)}\colon \Delta\in
      {\rm Reg}_{\geq t}\right\} \sum_{\Delta\in {\rm Reg}_{\geq t}} p(\Delta)  \\
 =  \min\left\{ \frac{\tilde{p}(\Delta)}{p(\Delta)}\colon \Delta\in
      {\rm Reg}_{\geq t}\right\} {\bf E}_\mu\left(\left|{\mathcal T}(S)_{\geq t}\right|\right). 
\end{multline}
  
We now estimate from below the quantity $\tilde{p}(\Delta)/p(\Delta)$ for $\Delta \in {\rm Reg}_{\geq t}$. Fix such a $\Delta$ and let ${\mathcal F}(\Delta)$ be the up-set on ground set $H \setminus {\mathcal I}_H(\Delta)$ generated by ${\mathcal D}(\Delta)$.
Using \textbf{(C3)} we see that
$$
p(\Delta) = (1-p)^{|{\mathcal I}_H(\Delta)|}\sum_{X \in {\mathcal F}(\Delta)} p^{|X|}(1-p)^{|H \setminus {\mathcal I}_H(\Delta)|-|X|}
$$
with an analogous expression for $\tilde{p}(\Delta)$. Recalling $\tilde{p}/p=1/t$ and that defining sets have size at most $s$, an application of Lemma \ref{lem-set_system_corr} immediately yields
\begin{align}
\frac{\tilde{p}(\Delta)}{p(\Delta)} & \geq \frac{(1-\tilde{p})^{|{\mathcal I}_H(\Delta)|}}{(1-p)^{|{\mathcal I}_H(\Delta)|}}\left(\frac{1}{t}\right)^s \nonumber \\
& \geq  \left(\frac{1-\tilde{p}}{1-p}\right)^{tn/r} \left(\frac{1}{t}\right)^s  \nonumber \\
& \geq  \left(\frac{e^{-c\tilde{p}}}{e^{-p}}\right)^{tn/r} \left(\frac{1}{t}\right)^s  \nonumber \\
& =  e^{t-c} t^{-s}, \label{eq:5}
\end{align} 
with the second inequality using $(1-\tilde{p})/(1-p) \geq 1$ and $|{\mathcal I}_H(\Delta)| \geq tn/r$, and the third inequality using the standard bound $1- p\leq e^{- p}$ (valid for all real $p$). In the third inequality we also use that for $0 \leq \tilde{p} \leq 1-\varepsilon$ (which certainly holds, since $\tilde{p} \leq p \leq 1-\varepsilon$) we have
$1-\tilde p  \geq e^{-c\tilde{p}}$ for some sufficiently large $c=c(\varepsilon)$ ($c=\log(1/\varepsilon)/(1-\varepsilon)$ will do).

Inserting (\ref{eq:5}) into (\ref{eq:4}) and combining with (\ref{eq:3}) we finally get
$$
{\bf E}_\mu\left(\left|{\mathcal T}(S)_{\geq t}\right|\right) \leq t^s e^{c-t} O(r^d/t^d) \leq C2^{-t}r^d
$$
for sufficiently large $C$. 
\end{proof}

%
%
%

\section{Optimal distal cell decomposition on the plane in o-minimal expansions of fields}
\label{sec: omin-cell-decomp}

Our goal in this section is to prove the following theorem.
\begin{thm}
  \label{thm:cell-dec}
Let $\CM$ be an o-minimal expansion of a real closed field. 
Then any  formula $\varphi(x;y)$ with $|x|=2$ admits a distal cell decomposition $\CT$ with 
$|\CT(S)| = O(|S|^2)$.
\end{thm}

Towards proving the theorem, we fix a formula $\varphi(x;y)$ with $|x|=2$ (and often we will write
$x$ as $(x_1,x_2)$).

We first construct a finite set of
formulas $\Phi(x;y)$ such that for any $s\in M^{|y|}$ the set
$\varphi(M;s)$ is a Boolean combination of $\Phi(x;s)$-definable sets,
and formulas in $\Phi(x;y)$ have a very simple form. Then we
construct a definable cell decomposition $\CT$ for
$\Phi(x;y)$ (hence also for $\varphi$)   with $|\CT(S)|=O(|S|^2)$.   

\medskip


Using o-minimality and definable choice  we can find  definable functions
$h_1,\dotsc,h_k\colon M\ttimes M^{|y|}\to M$ such that 
\[
h_1(a,s)\leq h_2(a,s)\leq \dotsb \leq h_k(a,s) \text{ for all } a\in
M,  s\in M^{|y|}, \]
and for all $a\in M, s\in M^{|y|}$ and $i=0,\dotsc,k$ we have 
\[ h_i(a,s) < x_1,x_1'  < h_{i+1}(a,s) \rightarrow [\varphi(x_1;a,s)\leftrightarrow
\varphi(x_1';a,s)], \]
 where for convenience we let $h_0(a,s)=-\infty$ and $
 h_{k+1}(a,s)=+\infty$.

\medskip

At this point we have that for  a fixed  $i=0,\dotsc,k$ for all $a\in
M$, $s\in M^{|y|}$ the truth value  of $\varphi(x_1;a, s)$ is constant on
the interval  $ h_i(a,s) < x_1 < h_{i+1}(a,s)$, but this constant truth value may depend on $a$. 
We need to partition $M$ into pieces where this truth
value does not depend on $a$.

For $a,a'\in M$ and $s\in M^{|y|}$ we define the relation $a
\sim_s a'$ as
\begin{multline*} a\sim_s a' \text{ iff  for all } i=0,\dotsc,k
 \\ \text{ and any } h_i(a,s) < x_1
< h_{i+1}(a,s),  \,  h_i(a',s) < x_1'  < h_{i+1}(a',s) \\\text{ we have } 
 \varphi(x_1;a,s)\leftrightarrow
\varphi(x_1';a',s).
\end{multline*}
Clearly $\sim_s$ is an equivalence relation on $M$ with at most
$2^{k+1}$-classes  uniformly definable in
terms of $s$. Using $o$-minimality and definable choice, we can find definable functions
$u_i\colon M^{|y|}\to M$, $i=1,\dotsc,l$ with $u_1(y)\leq u_2(y) \leq \dotsb\leq
u_l(y)$ such that for all $s\in M^{|y|}$ and $i=0,\dotsc,l$ we have 
$$ u_i(s)< x_2,x_2' <u_{i+1}(s) \rightarrow x_2\sim_s x_2',$$
where again for convenience we use $u_0(y)=-\infty$ and
$u_{l+1}(y)=+\infty$. 

We would prefer that for $s\in M^{|y|}$, each of the 
functions $x_2\mapsto h_i(x_2,s)$  was continuous. 
For $k \in \mathbb{N}$, we will write $[k]$ to denote the set $\{1,2, \ldots, k\}$. Since every definable function is piecewise continuous, we can further
partition $M$  and in addition require that for any $i=0,\dotsc,l$,
 $j\in[k]$  and
every $s\in M^{|y|}$ the function $x_2\mapsto h_j(x_2,s)$ is continuous on
the  interval $u_i(s) <x_2 < u_{i+1}(s)$. 
\medskip  

We take $\Phi(x;y)$ to be the following set of formulas (recall that $x=(x_1,x_2)$): 
\begin{gather*}
\{ x_2 = u_i(y) \colon i\in [l]\} \cup \{ x_2 < u_i(y) \colon i  \in [l] \}  \\
\cup \,  \{ x_2 > u_i(y) \colon i\in [l] \}\cup \{ x_1=h_i(x_2,y)\colon i\in [k]\} \\
\cup \, \{ x_1<h_i(x_2,y)\colon i\in [k]\} \cup \{ x_1>h_i(x_2,y)\colon i\in [k]\}. 
\end{gather*}

It is not hard to see that for any $s\in M^{|y|}$ the set $\varphi(M;s)$
is a Boolean combination of $\Phi(x;s)$-definable sets. 

We now
proceed with a construction of a  definable cell
decomposition for $\Phi(x;y)$. \\

Geometrically we view $M^2$ as $(x_1,x_2)$-plane, with $x_1$ being on the
vertical axis and $x_2$ on horizontal. Then 
$\Phi(x;S)$-definable sets partition the plain by  vertical lines
$x_2=u_i(s)$ and ``horizontal'' ``curves'' $x_1=h_j(x_2,s)$.

Unfortunately we cannot use complete $\Phi$-types over $S$ as
$\CT(S)$.  Since $S$ is finite every complete $\Phi$-type is equivalent
to a formula; however in general we cannot get uniform definability.

Consider a simple example of a partition of a plane by straight lines,
i.e.~the case when we don't have functions $u_i$ and 
have only one $h(x_2,a,b)$ defining the straight  
lines $x_1=ax_2+b$.  In the example below all points in the gray area have the same
$\Phi$-type, but we need at least $5$ lines to describe the region; and in general, this
number may be as big as one wants.\\

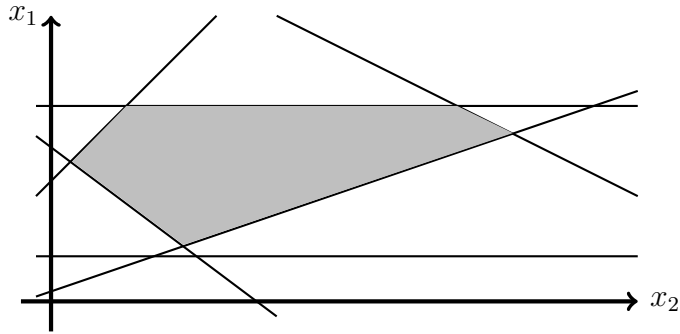
\begin{tikzpicture}[scale=0.4]
\draw[->,ultra thick] (-0.5,0.5)--(20,0.5) node[right]{$x_2$};
\draw[->,ultra thick] (0.5,-0.5)--(0.5,10) node[left]{$x_1$};
\draw[thick] (0,2)--(20,2);
\draw[thick] (0,4)--(6,10);
\draw[thick] (0,6)--(8,0);
\draw[thick] (0,7)--(20,7);
\draw[thick] (8,10)--(20,4);
\draw[thick] (0, 0.66)--(20,7.5); 
\fill[gray!50] (1.2,5.15)--(3,7) --(14,7)--(15.85,6.1)--(4.9,2.38)--cycle;
\end{tikzpicture}  

We could solve this problem by using also vertical lines  through all
points of intersections, as shown
below, but then the size of the partition would be $O(|S|^3)$. \\

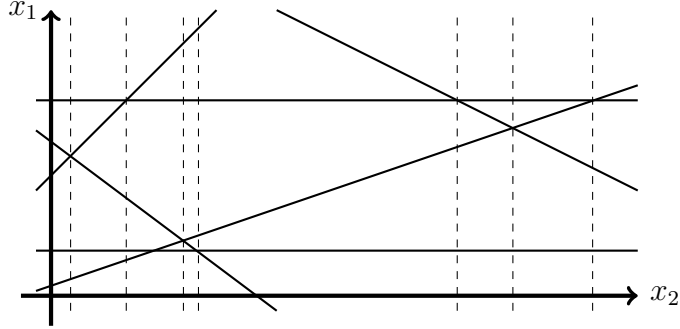
\begin{tikzpicture}[scale=0.4]
\draw[->,ultra thick] (-0.5,0.5)--(20,0.5) node[right]{$x_2$};
\draw[->,ultra thick] (0.5,-0.5)--(0.5,10) node[left]{$x_1$};
\draw[thick] (0,2)--(20,2);
\draw[thick] (0,4)--(6,10);
\draw[thick] (0,6)--(8,0);
\draw[thick] (0,7)--(20,7);
\draw[thick] (8,10)--(20,4);
\draw[thick] (0, 0.66)--(20,7.5); 
\draw[dashed] (1.15,0) --  (1.15,10);
\draw[dashed] (3,0) -- (3,10);
\draw[dashed] (4.9,0) -- (4.9,10);
\draw[dashed] (5.4,0) -- (5.4,10);
\draw[dashed] (14,0) -- (14,10);
\draw[dashed] (15.85,0)--(15.85, 10);
\draw[dashed] (18.5,0)--(18.5, 10);
\end{tikzpicture}  

Using the idea of ``vertical decomposition''  from \cite{13}
we  add only vertical line segments where they are
needed, i.e. from an intersection point to the first line above (or plus infinity) and
the first line below (or minus infinity), as in the following picture.

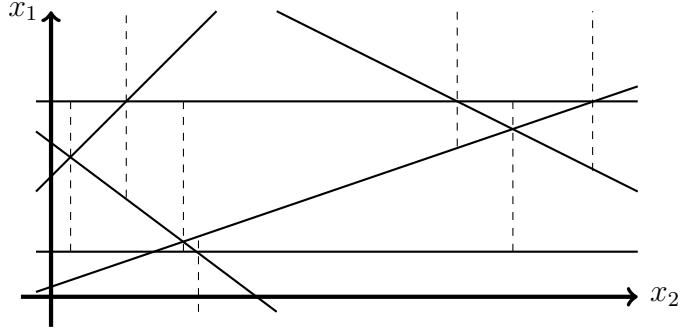
\begin{tikzpicture}[scale=0.4]
\draw[->,ultra thick] (-0.5,0.5)--(20,0.5) node[right]{$x_2$};
\draw[->,ultra thick] (0.5,-0.5)--(0.5,10) node[left]{$x_1$};
\draw[thick] (0,2)--(20,2);
\draw[thick] (0,4)--(6,10);
\draw[thick] (0,6)--(8,0);
\draw[thick] (0,7)--(20,7);
\draw[thick] (8,10)--(20,4);
\draw[thick] (0, 0.66)--(20,7.5); 
\draw[dashed] (1.15,2) --  (1.15,7);
\draw[dashed] (3,3.8) -- (3,10);
\draw[dashed] (4.9,2) -- (4.9,7);
\draw[dashed] (5.4,0) -- (5.4,2.55);
\draw[dashed] (14,5.5) -- (14,10);
\draw[dashed] (15.85,2)--(15.85, 7);
\draw[dashed] (18.5,4.69)--(18.5, 10);
\end{tikzpicture}  

Our  general case is slightly more complicated since the functions
$x_2\mapsto h_i(x_2,s)$ are not linear and even not continuous, just
piecewise continuous, so their graphs may intersect without one crossing
another.  \\

For $i\in [l]$ and $s\in M^{|y|}$ we will denote by $\hat u_i(s)$ the
corresponding vertical line 
\[ \hat u_i(s) :=\{ (x_1,x_2)\in M^2 \colon x_2=u_i(s)\},
 \]
and also for $i\in [k]$ and $s\in M^{|y|}$ we will denote by $\hat h_i(s)$
the ``curve''
\[ \hat h_i(s) :=\{ (x_1,x_2)\in M^2 \colon x_1=h_i(x_2,s)\}. \]

For $i,j\in [k]$, $s_1,s_2\in M^{|y|}$ and $(a,b)\in M^2$ we say that 
$\hat h_i(s_1)$ and $\hat h_j(s_2)$ \emph{properly intersect} at
$(a,b)$  if $(a,b)\in \hat h_i(s_1)\cap \hat h_j(s_2)$ and 
$\hat h_i(s_1),\hat h_j(s_2)$ have different germs at $(a,b)$.  
Formally it means that 
$a=h_i(b,s_1)=h_j(b,s_2)$  and for any $\varepsilon>0$ there is
$b'\in (b-\varepsilon, b+\varepsilon)$ with $h_i(b',s_1)\neq
h_j(b',s_2)$.  We will denote by $\hat h_i(s_1) \sqcap  \hat h_j(s_2)$
the set of all points $(a,b)\in M^2$ where   $\hat h_i(s_1)$ and  $
\hat h_j(s_2)$ intersect properly.  It is easy to see using
o-minimality that the set
$\hat h_i(s_1) \sqcap  \hat h_j(s_2)$
is finite and there is $N_l\in \NN$ such that 
$|\hat h_i(s_1) \sqcap \hat  h_j(s_2)|\leq N_l$ for all $i,j\in [k]$
and $s_1,s_2\in M^{|y|}$. Also all points in $\hat  h_i(s_1)\sqcap  \hat
h_j(s_2)$ are definable over $s_1,s_2$, i.e. there are definable
functions $f_{i,j}^m(y_1,y_2)$ with $m\in [N_l]$ such that for all
$s_1,s_2$ the set $\hat  h_i(s_1)\sqcap  \hat
h_j(s_2)$ is either empty or it is exactly $\{ f_{ij}^m(s_1,s_2) \colon
m\in [N_l] \}$.

\medskip
We will construct a definable cell decomposition $\CT(S)$ for 
$\Phi(x;y)$ as a union  of 5 families  of
cells: 
\begin{itemize}
\item $\CT_0(S)$ -- 0-dimensional cells, i.e. points; 
\item $\CT_1^u(S)$ -- 1-dimensional ``vertical'' cells;
\item $\CT_1^e(S)$ -- extra 1-dimensional vertical cells;
\item $\CT_1^h(S)$ -- 1-dimensional ``horizontal'' cells; 
\item $\CT_2(S)$ -- 2-dimensional cells. 
\end{itemize}
For each family $\CT_{\star}^\star(S)$ we will have
$|\CT_\star^\star(S)|=O(|S|^2)$, and also we will have  appropriate  $\Psi_\star^\star(x;\bar y)$
and $\CI_\star^\star(\Delta)$
so that 
\[\CT^\star_\star(S)=\{ \Delta \colon \Delta \text{ is $\Psi_\star^\star(x;S)$-definable and }
\CI_\star^\star(\Delta) \cap S =\emptyset\}.\]

In each case instead of defining the set of formulas $\Psi_\star^\star(x;\bar y)$, we describe  corresponding 
families  of $\Psi_\star^\star(x;S)$-definable sets, that we denote by $\Psi_\star^\star(S)$.

\medskip

\textbf{ The family $\mathbf{\CT_0(S)}$.}  We take $\CT_0(S)$ to be the
set of all points of intersections of vertical lines $\hat u_i(s)$ and
curves $\hat h_j(s')$ together with all points where 
curves $\hat h_i(s)$ and $\hat h_j(s')$ intersect properly. I.e.,
\begin{gather*}
\CT_0(S)=\cup \{ (\hat u_i(s)\cap \hat  h_j(s') \colon 
i\in [l];  j\in [k]; s,s'\in S\}\,   \\ 
\cup\, \{ \hat h_i(s) \sqcap \hat  h_j(s') \colon i,j\in [k]; 
s,s'\in S \}.
\end{gather*}

We take $\Psi_0(S) :=\CT_0(S)$ and $\CI_0(\Delta) :=\emptyset$. 

It is easy to see that  $\Psi_0(S)$ is uniformly definable. 

We also have $|\CT_0(S)|\leq k l|S|^2 + N_lk^2|S|^2 =O(|S|^2)$.

\medskip

\textbf{ The set $\mathbf{\CT_1^u(S)}$.}
For fixed  $i\in [l]$ and $s\in S$ let 
$I_i^s$ be the set of all definably connected components of $\hat
u_i(s) \setminus \CT_0(S)$.

Since 
\[ \hat u_i(s) \cap \CT_0(S)=\{ \hat u_i(s)\cap \hat h_j(s')\colon 
j\in [k], s'\in S\}, \]
we have $|I_i^s|\leq (k+1)|S|$, and every $\Delta\in  I_i^s$ has form 
\[ \Delta=\{ (x_1,x_2) \in M^2 \colon x_2=u_i(s); h_j(s_1) < x_1 < h_{j'}(s_2) \}, \]
for some $j,j'\in \{ 0,\dotsc, k+1\}$, and  $s_1,s_2\in S$. 

We take $\CT^u_1 (S)$ to be the union of all  $I_i^s$ for $i\in
[l]$ and $s\in S$.  Clearly $|\CT_1^u(S)|\leq l(k+1)|S|^2 =O(|S|^2)$.

We take $\Psi_1^u(S)$ to be the set of all vertical lines segments 
of the form 
$\{ (x_1,x_2) \in M^2 \colon x_2=u_i(s); h_j(s_1) < x_1 < h_{j'}(s_2)
\}$, for $i\in [l]$, $j,j'\in \{ 0,\dotsc, k+1\}$, $s,s_1,s_2\in S$.
For $\Delta\in \Psi_1^u(S)$  we take $\CI_1^u(\Delta) :=\{ s\in M^{|y|} \colon 
\Phi(x;s) \text{ crosses } \Delta\}$.

It is not hard to see that $\Psi_1^u$ and $\CI_1^u$ are uniformly
definable and $\CT_1^u(S)= \{ \Delta\in \Psi_1^u(S) \colon
\CI_1^u(\Delta)\cap S = \emptyset\}$.

\medskip
\textbf{ The set $\mathbf{\CT_1^e(S)}$.} For each point where two horizontal curves  intersect properly we add two vertical
line segments: one from the point to the curve above (or to plus infinity if there is no curve above) and one to the
curve below (or to minus infinity if there is no curve below).

Let  $i,j\in [k]$, $s,s_1\in S$ and $p=(p_1,p_2)\in \hat h_i(s) \sqcap
\hat h_j(s_1)$.

Let \[p^+ :=\inf \{ h_m(p_2,s') \colon m=1,\dotsc,k+1; s'\in S;
h_m(p_2,s')>p_1\},\]
and  
\[p^- :=\sup\{ h_m(p_2,s') \colon m=0,\dotsc,k; s'\in S;
h_m(p_2,s')<p_1\}.\]

We define $I_p^+ :=\{ (x_1,x_2)\in M^2 \colon x_2=p_2;\, p_1<x_1<p^+
\}$, $I_p^- :=\{ (x_1,x_2)\in M^2 \colon x_2=p_2;\, p^-<x_1<p_1
\}$; and take 
\[ \CT_1^e(S) := \{ I_p^+, I_p^- \colon p\in \hat h_i(s) \sqcap
\hat h_j(s_1); \, i,j\in [k]; \,  s,s_1\in S\}. \]

Obviously $|\CT_1^e(S)| \leq 2N_l k^2|S|^2 =O(|S|^2)$.

We take $\Psi_1^e(S)$ to be the family of all sets of the form 
\[ \{ (x_1,x_2)\in M^2 \colon x_2=p_2;\, p_1<x_1< h_m(p_2,s') 
\}\] 
for all $i,j\in [k]$, $m\in \{ 1,\dotsc,k+1\}$, $s,s_1,s'\in S$, and
$p=(p_1,p_2)\in   \hat h_i(s) \sqcap
\hat h_j(s_1)$; and of the form 
\[ \{ (x_1,x_2)\in M^2 \colon x_2=p_2;\, h_m(p_2,s')<x_1<  p_1
\}\] 
for all $i,j\in [k]$, $m\in \{ 0,\dotsc,k\}$, $s,s_1,s'\in S$, and
$p=(p_1,p_2)\in   \hat h_i(s) \sqcap
\hat h_j(s_1)$. It is not hard to see that $\Psi(S)$ is uniformly
definable. 

For $\Delta\in \Psi_1^e(S)$ we take $\CI_1^e(\Delta) :=\{ s\in M^{|y|} \colon
\Phi(x;s) \text{ crosses } \Delta\}$.  It is not hard to see
$\CI_1^e(\Delta)$ is uniformly
definable and $\CT_1^e(S)= \{ \Delta\in \Psi_1^e(S) \colon
\CI_1^e(\Delta)\cap S = \emptyset\}$.

\medskip

\textbf{ The set $\mathbf{\CT_1^h(S)}$.}
Given  $i\in [k]$ and $s\in S$, let $J_i^s$ be the set of all definably
connected components of $\hat h_i(s) \setminus \CT_0(S)$. 
It is easy to see that  
\begin{multline*}
 \hat h_i(s)\cap \CT_0(S)=\{ \hat h_i(s)\cap \hat u_j(s')\colon j\in
[l]; s'\in S \}  \\
\cup\, \{ \hat h_i(s)\sqcap \hat h_j(s')\colon j\in
[k]; s'\in S \}. 
\end{multline*}
In particular $|J_i^s|\leq (l+N_l k+1)|S|$. 

We take $\CT_1^h(S)$ to be the union of all $J_i^s$ for $i\in [k]$,
$s\in S$.   Clearly $|\CT_1^h (S)|\leq k(l+N_l k+1)|S|^2=O(|S|^2)$.

Given  $i\in [k]$, $s\in S$ and $s_1,s_2\in S$ let   
$\mathcal A_{i,s}[s_1,s_2]$  be the family of all sets of the form $\{ (x_1,x_2)\in \hat h_i(s); c_1<x_2 < c_2
  \},$ with 
\begin{multline*}
 c_1,c_2\in \{ u_j(s_1) \colon j\in [l]\}\\ 
\cup \,\{ p_2
  \colon (p_1,p_2) \in \hat h_i(s) \sqcap \hat h_j(s_2) \text{ for some }
  p_1\}\cup
\{ \pm\infty\}.
\end{multline*}

We take $\Psi_1^h(S)$ to be the union of all $\mathcal A_{i,s}[s_1,s_2]$ with $i\in
[k]$ and $s,s_1,s_2\in S$.
 It is not hard to see that 
$\Psi_1^h(S)$ is uniformly
definable and $\CT_1^h(S)= \{ \Delta\in \Psi_1^h(S) \colon
\CI_1^h(\Delta)\cap S = \emptyset\}$, where $\CI_1^h(\Delta)=\{ s\in M^{|y|} \colon
\Phi(x;s) \text{ crosses } \Delta\}$.

\medskip

\textbf{ The set $\mathbf{\CT_2(S)}$.} For the family $\CT_2(S)$ we
take all definably connected components of $M^2\setminus
(\CT_0(S)\cup
\CT_1^u(S)\cup
\CT_1^e(S)\cup
\CT_1^h(S))$.

Given  $i,j\in \{0,\dotsc,k+1\}$ , $s_1,s_2\in S$   and  $c_1<c_2\in
M\cup\{\pm\infty\}$ 
with $h_i(x_2,s_1)< h_j(x_2,s_2)$ for all $x_2\in (c_1,c_2)$,
let
$A_{i,s_1}^{j,s_2}(c_1,c_2)$  be the set 
\[ A_{i,s_1}^{j,s_2} (c_1,c_2) =\{ (x_1,x_2)\in M^2\colon c_1<x_2 < c_2;\,
h_i(x_2,s_1)< x_1< h_j(x_2,s_2)  \}.  \]

It is not hard to see that if  $\Delta\in \CT_2(S)$ then $\Delta=
A_{i,s_1}^{j,s_2} (c_1,c_2)$ for some $i,j\in\{0,\dotsc,k+1\}$,
 $s_1,s_2\in S$ and $c_1,c_2$  belonging to the following
set:
\begin{gather*}S_{i,s_1}^{j,s_2}=\{  u_{i'}(s') \colon i'\in \{ 0,\dotsc l+1\}; s'\in
  S\} \\ 
\cup\,\{ p_2 \colon (p_1,p_2)\in \hat h_i(s_1)\sqcap \hat h_{i'}(s') 
\text{ for some } i'\in [k], s'\in S, p_1\in M \} \\
\cup\,\{ p_2 \colon (p_1,p_2)\in \hat h_j(s_2)\sqcap \hat h_{i'}(s') 
\text{ for some } i'\in [k], s'\in S, p_1\in M \}. 
\end{gather*}

We take $\Psi_2(S)$ to be the family of all $A_{i,s_1}^{j,s_2}
(c_1,c_2)$, for all $c_1,c_2\in S_{i,s_1}^{j,s_2}$. 

It is not hard to see that $\Psi_2(S)$ is uniformly definable family,
and we have $\CT_2(S) \subseteq \Psi_2(S)$.

It is also not hard to see that a set $\Delta\in \Psi_2(S)$ is in
$\CT_2(S)$ if and only if it is not crossed by $\Phi(x;S)$, and is also not crossed
by any line segment in $\CT_1^e(S)$.

Hence a set $\Delta=A_{i,s_1}^{j,s_2}(c_1,c_2)\in \Psi_2(S)$ \textbf{ is not in } $\CT_2(S)$ if and
only if there is $s\in S$ satisfying at least one of the following conditions. 
\begin{enumerate}[(C1)]
\item  $\Phi(x;s)$ crosses $\Delta$. 
\item There are  $i'\in [k]$ and 
  $(p_1,p_2)\in  \hat h_i(s_1) \sqcap \hat h_{i'}(s)$ with
  $c_1<p_2<c_2$.
\item There are  $i'\in [k]$ and 
  $(p_1,p_2)\in  \hat h_j(s_2) \sqcap \hat h_{i'}(s)$ with
  $c_1<p_2<c_2$.
\end{enumerate}

For $\Delta\in \Psi_2(S)$ we take $\CI_2(\Delta)$ to be the set of all
$s\in M^{|y|}$ satisfying any of the conditions $(C1)-(C3)$. It is not
hard to see that $\CI_2(\Delta)$ is uniformly definable and
$\CT_2(S)=\{ \Delta\in \Psi_2(S) \colon \CI_2(\Delta)\cap S
=\emptyset\}$. 

\medskip

We are left to check that $|\CT_2(S)|=O(|S|^2)$. 

Since $\CT_2(S)$ consists of definably connected components  of 
 $M^2\setminus
(\CT_0(S)\cup
\CT_1^u(S)\cup
\CT_1^e(S)\cup
\CT_1^h(S))$, any two $\Delta,\Delta'\in \CT_2(S)$ are either disjoint
or coincide, hence every $\Delta\in \CT_2(S)$ is completely determined
by its ``left lower corner'', i.e. if   $\Delta=A_{i,s_1}^{j,s_2}(c_1,c_2)$
and $\Delta'=A_{i,s_1}^{j',s'_2}(c_1,c'_2)$ are in $\CT_2$ then
$\Delta=\Delta'$.

We divide $\CT_2(S)$ into 4 disjoint families: 
\begin{itemize}
\item The family $F_1(S)$ of all $A_{i,s_1}^{j,s_2}(c_1,c_2)\in \CT_2(S)$
  with $c_1=-\infty$.

\item  The family $F_2(S)$ of all $A_{i,s_1}^{j,s_2}(c_1,c_2)\in \CT_2(S)$
  with $c_1=u_{i'}(s')$ for some $i' \in [l]$ and $s'\in S$.
\item  The family $F_3(S)$ of all $A_{i,s_1}^{j,s_2}(c_1,c_2)\in \CT_2(S)$ 
that are not in $F_2(S)$ and $(p_1,c_1)\in \hat h_i(s_1)\sqcap \hat
h_{i'}(s')$  for some $i' \in [k]$, $s'\in S$, and $p_1\in  M$. 

\item  The family $F_4(S)$ of all $A_{i,s_1}^{j,s_2}(c_1,c_2)\in \CT_2(S)$ 
that are not in $F_1(S)\cup F_2(S)\cup F_3(S)$.  In this case we have that $\{ (x_1,c_1) \colon h_i(c_1,s_1)< x_1
< h_j(c_1,s_2) \}\in \CT_1^e(S)$.
\end{itemize}

Every  $A_{i,s_1}^{j,s_2}(c_1,c_2)\in F_1(S)$ is completely determined by
$i$ and $s_1$, hence $|F_1(S)|\leq (k+1)|S|$ (we get $k+1$, since we
allow $i=0$). 

Every $A_{i,s_1}^{j,s_2}(c_1,c_2)\in F_2(S)$ is completely determined by
$i$, $s_1$, some $i'\in [l]$ and $s'\in S$. Hence 
$|F_2(S)|\leq (k+1)l|S|^2$.

Since $ \hat h_i(s_1)\sqcap \hat
h_{i'}(s')\leq N_l$ we have $|F_3(S)|\leq k^2N_l|S|^2$.

Finally,  each $A_{i,s_1}^{j,s_2}(c_1,c_2)\in F_4(S)$  is completely
determined by its ``left side'' $\{ (x_1,c_1) \colon h_i(c_1,s_1)< x_1
< h_j(c_1,s_2) \}$ that is in $\CT_1^e(S)$. Since
$|\CT_1^e(S)|=O(|S|^2)$, we also have $|F_4(S)|=O(|S|^2)$.

Therefore $|\CT_2(S)|=O(|S|^2)$. 

\medskip

Taking $\CT(S)=\CT_0(S)\cup
\CT_1^u(S)\cup
\CT_1^e(S)\cup
\CT_1^h(S)\cup \CT_2(S)$ we obtain a definable cell decomposition
for $\Phi(x;y)$ with $|\CT(S)|=O(|S|^2)$. 

\section{Planar Zarankiewicz's problem in distal structures}\label{sec: Zarank}

\subsection{Zarankiewicz's problem}

Zarankiewicz's problem in graph theory asks to determine the largest possible number of edges in a bipartite graph on a given number of vertices that has no complete bipartite subgraphs of a given size.

In \cite{zaran} the authors investigate Zarankiewicz's problem for semialgebraic graphs of bounded description complexity, a setting which in particular subsumes a lot of different incidence-type questions.

In particular, they prove the following upper bound on the number of edges (they have more general results in $\mathbb{R}^n$ for arbitrary $n$ as well, but here we will be only concerned with the ``planar'' case). 
\begin{fact} \cite[Theorem 1.1]{zaran}\label{fac: semialg zarank}
Let $E \subseteq \mathbb{R}^2 \times \mathbb{R}^2$ be a semi-algebraic relation such that $E$ has description complexity at most $t$ (i.e., $E$ can be defined as a Boolean combination of at most $t$ polynomial inequalities, with all of the polynomials involved of degree at most $t$). Then for any $k \in \mathbb{N}$ there is some constant $c = c(t,k)$ satisfying the following.

If $P, Q \subseteq \mathbb{R}^2$ with $|P|=m, |Q|=n$ are such that $E \cap (P \times Q)$ doesn't contain a copy of $K_{k,k}$ (the complete bipartite graph with both parts of size $k$), then
$$|E(P,Q)| \leq c \left( (mn)^{\frac{2}{3}} + m + n \right),$$

where $E(P,Q) = E \cap (P \times Q)$.

\end{fact}

\begin{rem} This result is a natural generalization of the Szemer\'edi-Trotter theorem over $\RR$ \cite{szemeredi1983extremal}. Namely, if $P$ a set of points on the plane, $Q$ the dual of the lines (i.e.~lines are semi-algebraically coded by points in $\mathbb{R}^2$), and $E$ the incidence relationship (which is also clearly semialgebraic), then $E(P,Q)$ is $K_{2,2}$-free as any two distinct lines intersect in at most one point.
\end{rem}

We will give a common generalization of Fact \ref{fac: semialg zarank} and the semialgebraic ``points / planar curves'' incidence bound from \cite[Theorem 4]{pach1992repeated} to arbitrary definable families admitting a quadratic distal cell decomposition (e.g.~any definable family of subsets of $M^2$ in an o-minimal expansion of a field). To state the result, we first recall the notion of the VC-density of a partitioned formula (and refer to \cite{VCD1} for a detailed discussion).

\begin{defn}
\begin{enumerate}
\item Given a set $X$ and a family $\CF$ of subsets of $X$, the \emph{shatter function} $\pi_{\CF}: \mathbb{N} \to \mathbb{N}$ of $\CF$ is defined as
$$ \pi_{\CF}(n) := \max \{ |\CF \cap A| : A \subseteq X, |A| = n \}, $$
where $\CF \cap A = \{ S \cap A : S \in \CF \}$.

\item The \emph{VC-density} of $\CF$, or $\vc(\CF)$, is defined as the infimum of all real numbers $r$ such that $\pi_{\CF}(n) = O(n^r)$ (and $\vc(\CF) = \infty$ if there is no such $r$).

\item Given a formula $\varphi(x;y)$, possibly with parameters from $M$, we let $\mathcal{F}_{\varphi(x;y)} :=  \{ \varphi(M;b) : b \in M^{|y|}\}$ be the family of all $\varphi$-definable subsets of $M^{|x|}$.

\item We define the VC density of $\varphi$ to be $\vc(\varphi) := \vc(\mathcal{F}_{\varphi})$.

\item Given a formula $\varphi(x;y)$, we consider its dual formula $\varphi^*(y;x) := \varphi(x;y)$ obtained by interchanging the roles of the variables. It is easy to see then that the family $\mathcal{F}_{\varphi^*(y,x)} = \{ \varphi^*(M; a) : a \in M^{|x|}\}$ of subsets of $M^{|y|}$ is the dual set system for the family $\{ \varphi(M; b) : b \in M^{|y|}\}$ of subsets of $M^{|x|}$.
\end{enumerate}

\end{defn}

VC-density in various classes of NIP structures is investigated e.g. in \cite{VCD1, VCD2}, and the optimal bounds are known in some cases including the o-minimal structures.

\begin{fact} \cite[Theorem 6.1]{VCD1}\label{fac: vc bound in o-min}
Let $\CM$ be an o-minimal structure, and let $\varphi(x;y)$ be any formula. Then $\vc(\varphi^*) \leq |x|$.
\end{fact}

\begin{rem}\label{rem: distal decomp bounds dual VC density}
Let $\varphi(x;y)$ be a formula admitting a distal cell decomposition $\CT$ with $|\CT(S)| = O(|S|^d)$. Then $\vc(\varphi^*) \leq d$. 

Indeed, recalling Definition \ref{def: def cell decomp}, given any finite $S \subseteq M^{|y|}$ and $\Delta \in \CT(S)$, $S \cap \varphi^*(M,a) = S \cap \varphi^*(M,a')$ for any $a,a' \in \Delta$ (and the sets in $\CT(S)$ give a covering of $M^{|x|}$), hence at most $|S|^d$ different subsets of $S$ are cut out by the instances of $\varphi^*(y;x)$.
\end{rem}

We will need the following weaker bound that applies to graphs of bounded VC-density.

\begin{fact} \label{VCBoundOnEdges}
\cite[Theorem 2.1]{zaran} For every $\alpha \in \mathbb{R}$ and  $d, k \in \mathbb{N}$ there is some constant $\alpha_1 = \alpha_1(\alpha,d,k)$ such that the following holds.

  Let $E \subseteq P \times Q$ be a bipartite graph with $|P|=m, |Q|=n$ such that the family of sets $\CF = \{ E(q) : q \in Q \}$ satisfies $\pi_{\mathcal{F}}(z) \leq \alpha z^d$ for all $z \in \mathbb{N}$ (where $E(q) = \{ p \in P : (p,q) \in E \}$). Then if $E$ is $K_{k,k}$-free, we have 
$$|E(P,Q)| \leq \alpha_1(m n^{1- 1/d}+n).$$
\end{fact}

%

We are ready to prove the main theorem of this section.

\begin{thm}\label{thm: distal Zarank}
Let $\CM$ be a structure and $d,t \in \mathbb{N}_{\geq 2}$. Assume that $E(x,y) \subseteq M^{|x|} \times M^{|y|}$ is a definable relation given by an instance of a formula $\theta(x,y;z) \in \mathcal{L}$, such that the formula $\theta'(x;y,z) := \theta(x,y;z)$ admits a distal cell decomposition $\CT$ with    $|\CT(S)| = O(|S|^t)$ and such that $\vc(\theta'') \leq d$ for $\theta''(x,z;y) := \theta(x,y;z)$. Then for any $k \in \mathbb{N}$ there is a constant $\alpha = \alpha(\theta,k)$ satisfying the following.

For any finite $P \subseteq M^{|x|}, Q \subseteq M^{|y|}$, $|P|=m, |Q| = n$, if $E(P,Q)$ is $K_{k,k}$-free, then we have:
\begin{equation} \label{eq: Zar}
|E(P,Q)| \leq \alpha \left( m^{\frac{(t-1)d}{td-1}} n^{\frac{td-t}{td-1}} + m + n \right).
\end{equation}
\end{thm}

\begin{proof} 

Our argument is a generalization of the proofs of \cite[Theorem 3.2]{zaran} and \cite[Theorem 4]{pach1992repeated}.

Let $E(x,y) = \theta(x,y;c^*) = \theta'(x;y,c^*) = \theta''(x,c^*;y)$ for a given tuple of parameters $c^* \in M^{|z|}$. Note that for any $n \in \mathbb{N}$ we clearly have $\pi_{\mathcal{F}_{E(x,y)}}(n) = \pi_{\mathcal{F}_{\theta''(x,c^*;y)}} \leq  \pi_{\mathcal{F}_{\theta''(x,z;y)}}(n)$. By assumption $\vc(\theta''(x,z;y)) \leq d$, hence there is some $\alpha_0 = \alpha_0(\theta)$ such that $\pi_{\mathcal{F}_{E(x,y)}}(n) \leq \alpha_0 n^d$.

If $n \geq m^d$, then by Fact \ref{VCBoundOnEdges} we have 
$$|E(P,Q)| \leq \alpha_1 ( m n^{1 - \frac{1}{d}} + n) \leq \alpha_1 (n^{\frac{1}{d}} n^{1 - \frac{1}{d}} + n) =  2\alpha_1 n$$ 
for some $\alpha_1 = \alpha_1(\theta,d,k)$, and we are done.
Hence we assume $n < m^d$.

Let $r := \frac{m^{\frac{d}{td-1}}}{n^{\frac{1}{td-1}}}$ (note that $r > 1$ as $m^d > n$), and consider the family $\Sigma = \{E(M,q) : q \in Q\}$ of subsets of $M^{|x|}$. 

By assumption and Theorem \ref{lem-r_cut} (and Remark  \ref{rem: triv cut r geq n} in the case $r \geq n$) applied to the formula $\theta'(x; y,z)$ and the collection of parameters $H := \{ (q,c^*) \in M^{|y|} \times M^{|z|} : q \in Q\}$, there is a family $\mathcal{C}$ of subsets of $M^{|x|}$ giving a $\frac{1}{r}$-cutting for the family $\Sigma$.
That is, $M^{|x|}$ is covered by the union of the sets in $\mathcal{C}$ and any of the sets $C \in \mathcal{C}$ is crossed by at most $|\Sigma|/r$ elements from $\Sigma$. Moreover, $|\mathcal{C}| \leq \alpha_2 r^t$ for some $\alpha_2 = \alpha_2(\theta)$.

Then there is a set $C \in \mathcal{C}$ containing at least $\frac{m}{\alpha_2 r^t} = \frac{n^{\frac{t}{td-1}}}{\alpha_2 m ^{\frac{1}{td-1}}}$ points from $P$. Let $P' \subseteq P \cap C$ be a subset of size exactly $\left\lceil \frac{n^{\frac{t}{td-1}}}{\alpha_2 m ^{\frac{1}{td-1}}} \right\rceil$.

If $|P'| < k$, we have $\frac{n^{\frac{t}{td-1}}}{\alpha_2 m ^{\frac{1}{td-1}}} \leq |P'| < k$, so $n < k^{\frac{td-1}{t}} \alpha_2^{\frac{td-1}{t}}m^{\frac{1}{t}}$.

Note that $\pi_{\mathcal{F}_{E^*(y,x)}}(n)  = \pi_{\mathcal{F}_{(\theta')^*(y,c^*;x)}}(n)   \leq \pi_{\mathcal{F}_{(\theta')^*(y,z;x)}}(n) \leq  \alpha_3 n^t$ for some $\alpha_3 = \alpha_3(\theta)$, where the last inequality holds by Remark \ref{rem: distal decomp bounds dual VC density} applied to the formula $\theta'(x; y,z)$. Then by  Fact \ref{VCBoundOnEdges} applied to the relation $E^*$  we have 
$$|E(P,Q)| \leq \alpha_4 (n m^{1 - \frac{1}{t}} + m) \leq \alpha_4 (k^{\frac{td-1}{t}} \alpha_2^{\frac{td-1}{t}}m^{\frac{1}{t}} m^{1 - \frac{1}{t}} + m) \leq \alpha_5 m$$
for some $\alpha_5 = \alpha_5 (\theta,k)$, so we are done.

Hence we may assume that $|P'| \geq k$. 
Let $Q'$ be the set of all points $q \in Q$ such that $E(M,q)$ crosses $C$. We know that 
$$|Q'| \leq \frac{|Q|}{r} \leq  \frac{n n^{\frac{1}{td-1}}}{m^{\frac{d}{td-1}}} = \frac{n^{\frac{td}{td-1}}}{m^{\frac{d}{td-1}}} \leq \alpha_2^d |P'|^d.$$
Again by Fact \ref{VCBoundOnEdges} we get 
$$|E(P',Q')| \leq \alpha_1(|P'| |Q'|^{1-\frac{1}{d}} + |Q'|)$$
$$ \leq \alpha_1 (|P'|\alpha_2^{d-1} |P'|^{d-1} + \alpha_2^d |P'|^d) \leq \alpha_6 |P'|^d$$
for some $\alpha_6 = \alpha_6(\theta,k)$. Hence there is a point $p \in P'$ such that $|E(p) \cap Q'| \leq \alpha_6 |P'|^{d-1}$.

Since $E(P,Q)$ is $K_{k,k}$-free, there are at most $k-1$ points in $Q\setminus Q'$ from $E(p)$ (otherwise, since none of those points crosses $C$ and $C$ contains $P'$, which is of size $\geq k$, we would have a copy of $K_{k,k}$). And we have $|P'| \leq  \frac{n^{\frac{t}{td-1}}}{\alpha_2 m ^{\frac{1}{td-1}}} + 1  \leq \frac{2}{\alpha_2} \frac{n^{\frac{t}{td-1}}}{ m ^{\frac{1}{td-1}}}$ as $|P|' \geq k \geq 1$. Hence 
$$|E(p)| \leq \alpha_6 |P'|^{d-1} + (k-1) \leq \alpha_7 \frac{n^{\frac{t(d-1)}{td-1}}}{m^{\frac{d-1}{td-1}}} + (k-1)$$
for $\alpha_7 := \frac{\alpha_6 2^{d-1}}{\alpha_2^{d-1}}$. We remove $p$ and repeat the argument until we have no vertices remaining in $P$, and see that 
$$|E(P,Q)| \leq (2 \alpha_1 + \alpha_5)(n+m) + \sum_{i=n^{\frac{1}{d}}}^{m} \left( \alpha_7 \frac{n^{\frac{t(d-1)}{td-1}}}{i^{\frac{d-1}{td-1}}} + (k-1) \right)$$
$$ \leq (2 \alpha_1 + \alpha_5)(n+m) + \alpha_7 n^{\frac{t(d-1)}{td-1}} \sum_{i=n^{\frac{1}{d}}}^{m} \frac{1}{i^{\frac{d-1}{td-1}}} + (k-1) m .$$
Note that 
$$\sum_{i=n^{\frac{1}{d}}}^{m} \frac{1}{i^{\frac{d-1}{td-1}}} \leq \int_{n^{\frac{1}{d}} -1}^{m} \frac{dx}{x^{\frac{d-1}{td-1}}} = \frac{m^{1 - \frac{d-1}{td-1}}}{1-\frac{d-1}{td-1}} - \frac{\left(n^{\frac{1}{d}} -1 \right)^{1 - \frac{d-1}{td-1}}}{1 - \frac{d-1}{td-1}}$$
$$ \leq \frac{td-1}{(t-1)d} m^{1 - \frac{d-1}{td-1}}$$
using $d,t \geq 2$, for all $n$ large enough with respect to $d$ (as the second term is positive then). Hence we can choose $\alpha = \alpha(\theta,k)$ large enough so that
$$|E(P,Q)| \leq \frac{\alpha}{3}(n+m) + \frac{\alpha}{3} n^{\frac{t(d-1)}{td-1}} m^{1-\frac{d-1}{td-1}} + \frac{\alpha}{3}m$$
$$ \leq \alpha ( m^{\frac{(t-1)d}{td-1}} n^{\frac{td-t}{td-1}} + m + n)$$
for all $m,n$.
\end{proof}

\begin{rem}
	In a different regime, one can consider the situation when $E$ admits a distal cell decomposition of exponent $t$, but instead of bounding the dual VC-density by $d$, we assume that $K_{s,d}$ is omitted. Then same bound as in \eqref{eq: Zar} holds, up to terms of smaller order, with the constant $\alpha$ depending only on $s,d, \theta$ --- see \cite{ESstrmin} for the details.
	
\end{rem}

\subsection{Omitting $K_{k,k}$ versus omitting infinite complete bipartite graphs}
We recall a result of Bukh and Matou\v{s}ek.

\begin{fact}\cite[Theorem 1.9]{bm}\label{fac: semialg BM}
For every $d,D$ and $k$ there exists $N$ such that for every semialgebraic relation $R(x_1, \ldots, x_k)$ with $|x_1| = \ldots = |x_k|=d$ of description complexity $D$, the following two conditions are equivalent.
\begin{enumerate}
\item There exist $A_1, \ldots, A_k \subseteq \mathbb{R}^d$ such that $|A_1| = \ldots = |A_k| = N$ and $A_1 \times \ldots \times A_k \subseteq R$.
\item There exist \emph{infinite} sets $A_1, \ldots, A_k \subseteq \mathbb{R}^d$ such that $A_1 \times \ldots \times A_k \subseteq R$.
\end{enumerate}

\end{fact}

We give a generalization of this result for any distal structure in which finite sets in every definable family have a uniform bound on their size.
Recall:

\begin{defn}\label{def: elimination of infinity}
An $\mathcal{L}$-structure $\mathcal{M}$ \emph{eliminates $\exists^\infty$} if for every $\varphi(x,y) \in \mathcal{L}$ there is some $n_\varphi \in \mathbb{N}$ such that for any $b \in M^{|y|}$, $\varphi(M,b)$ is infinite if and only if $|\varphi(M,b)| \geq n_\varphi$.
\end{defn}

We will use the definable strong Erd\H{o}s-Hajnal property for hypergraphs in distal structures from \cite{distal} (and we will use some terminology from that paper in our argument).
\begin{fact}\cite[Corollary 4.6]{distal}\label{fac: definable strong EH in distal}
Let $\mathcal{M}$ be a distal $\mathcal{L}$-structure. Then for every formula $\varphi(x_1, \ldots, x_k; z) \in \mathcal{L}$ there are some $\alpha > 0$ and formulas $\psi_i(x_i, y_i) \in \mathcal{L}$ for $1 \leq i \leq k$ such that the following holds.

For any generically stable Keisler measures $\mu_i$ on $M^{|x_i|}$ and any $c \in M^{|z|}$, there are some $b_i \in M^{|y_i|}$ such that $\mu_i(\psi_i(M^{|x_i|},b_i)) \geq \alpha$ and either 
$$\prod_{1\leq i \leq k} \psi_i(M^{|x_i|},b_i) \subseteq \varphi(M^{|x_1|}, \ldots, M^{|x_k|}; c) \textrm{, or}$$ 
$$\prod_{1\leq i \leq k} \psi_i(M^{|x_i|},b_i) \subseteq \neg \varphi(M^{|x_1|}, \ldots, M^{|x_k|}; c).$$
\end{fact}

\begin{thm}\label{thm: distal BukhMatousek}
Let $\mathcal{M}$ be a distal $\mathcal{L}$-structure eliminating $\exists^{\infty}$. Then for any formula $\varphi(x_1, \ldots, x_k;z) \in \mathcal{L}$ there is some $N \in \mathbb{N}$ and $\psi_i(x_i,y_i) \in \mathcal{L}$, for $1\leq i \leq k$, such that the following are equivalent for any $c \in M^{|z|}$, letting $R \subseteq M^{|x_1|} \times \ldots \times M^{|x_k|}$ be given by $R := \varphi(M^{|x_1|}, \ldots, M^{|x_k|}, c)$.

\begin{enumerate}
\item There exist $A_i \subseteq M^{|x_i|}$ for $1\leq i \leq k$ such that $|A_1| = \ldots = |A_k| = N$ and $A_1 \times \ldots \times A_k \subseteq R$.
\item There are some $b_i \in M^{|y_i|}$ such that $\psi_i(M^{|x_i|}, b_i)$ is infinite for all $1 \leq i \leq k$ and $\psi_1(M^{|x_1|}, b_1) \times \ldots \times \psi_k(M^{|x_k|}, b_k) \subseteq R$.
\end{enumerate}

\end{thm}
\begin{proof}
Let $\alpha > 0$ and $\psi_i(x_i,y_i) \in \mathcal{L}$, for $1\leq i \leq k$, be as given by Fact \ref{fac: definable strong EH in distal} for $\varphi(x_1, \ldots, x_k;z)$.
Let $n_i \in \mathbb{N}$ be as given by Definition \ref{def: elimination of infinity} for $\psi_i(x_i,y_i)$, and let $n := \max \{ n_i : 1 \leq i \leq k \}$.
We take $N := \lceil \frac{n}{\alpha} \rceil$, then $N = N(\varphi)$.

Let $c \in M^{|z|}$ be arbitrary, and let $R := \varphi(M^{|x_1|}, \ldots, M^{|x_k|}, c)$. Assume that (1) holds. That is, there are some $A_i \subseteq M^{|x_i|}$ such that $|A_1| = \ldots = |A_k| = N$ and $A_1 \times \ldots \times A_k \subseteq R$. Let $\mu_i$ be a Keisler measure on $M^{|x_i|}$ defined by $\mu_i(X) := \frac{|A_i \cap X|}{|A_i|}$ for all definable $X \subseteq M^{|x_i|}$, then $\mu_i$ is generically stable for all $1 \leq i \leq k$. Applying Fact \ref{fac: definable strong EH in distal}, we find some $b_i \in M^{|y_i|}$ such that $\mu_i(\psi_i(M^{|x_i|}, b_i)) \geq \alpha$ and $\prod_{1\leq i \leq k} \psi_i(M^{|x_i|},b_i) \subseteq R$ (note that $\prod_{1\leq i \leq k} \psi_i(M^{|x_i|},b_i) \subseteq \neg R$ is impossible as $\prod_{1 \leq i \leq k} A_i \subseteq R$). Now for any $1\leq i \leq k$, $\mu_i(\psi_i(M^{|x_i|}, b_i)) \geq \alpha$ implies $|\psi_i(A_i, b_i)| \geq \alpha N \geq n_i$, hence $\psi_i(M^{|x_i|}, b_i)$ is infinite by the choice of $n_i$, as wanted.
\end{proof}

\begin{rem}
Examples of structures satisfying the assumption of Theorem \ref{thm: distal BukhMatousek} are given by arbitrary o-minimal structures and $p$-minimal structures (e.g. the field $\mathbb{Q}_p$). Hence Fact \ref{fac: semialg BM} follows by applying it to the field of reals.
\end{rem}
\subsection{The o-minimal case}

Theorem \ref{thm: distal BukhMatousek} implies that in Theorem \ref{thm: distal Zarank}, assuming $\mathcal{M}$ eliminates $\exists^\infty$, we can relax the assumption to just assuming that $E$ doesn't contain a copy of an infinite complete bipartite graph. We conclude by observing that all of these results apply to o-minimal expansions of fields.

\begin{thm}\label{thm: everything o-min}

Let $\CM$ be an o-minimal expansion of a field and let $E(x,y) \subseteq M^2 \times M^d$ be a $\theta$-definable relation. 

\begin{enumerate}
\item For every $k \in \mathbb{N}$ there is a constant $\alpha = \alpha(\theta,k)$ such that for any finite $P \subseteq M^2, Q \subseteq M^d$, $|P|=m, |Q| = n$, if $E(P,Q)$ does not contain a copy of $K_{k,k}$ (the complete bipartite graph with two parts of size $k$), then we have 
	$$ |E(P,Q)| \leq \alpha \left( m^{\frac{d}{2d-1}} n^{\frac{2d-2}{2d-1}} + m + n \right).$$

	\item There is some $k' \in \mathbb{N}$ and formulas $\varphi(x,v), \psi(y,w)$, all depending only on $\theta$,  such that if $E$ contains a copy of $K_{k',k'}$, then there are some parameters $b \in M^v, c \in M^w$ such that both $\varphi(M,b)$ and $ \psi(M,c)$ are infinite and $\varphi(M,b) \times \psi(M,c) \subseteq E$.
	
		\end{enumerate}
\end{thm}

\begin{proof}
(1) Follows by applying Theorem \ref{thm: distal Zarank}. Its assumptions are satisfied for an arbitrary formula $\theta(x,y;z)$ with $|x| = 2$ and $|y| = d$ by Theorem \ref{thm:cell-dec} applied to $\theta'(x; y,z)$ and by Fact \ref {fac: vc bound in o-min} applied to the dual formula $(\theta'')^*(x,z ; y)$.

(2) Follows by Theorem \ref{thm: distal BukhMatousek} as o-minimal theories eliminate the $\exists^\infty$ quantifier.
\end{proof}

\begin{rem}
	Theorem \ref{thm: distal Zarank} could be used to obtain a Zarankiewicz-type bound for definable relations $E \subseteq M^t \times M^d$ in $o$-minimal structures, with $t \in \mathbb{N}$ arbitrary. However, we don't pursue it here since optimal bounds for distal cell decompositions are not known  for $t > 2$.
\end{rem}

\begin{cor}
In the setting of Theorem \ref{thm: everything o-min}, there is a constant $\alpha$ and formulas $\varphi(x,v), \psi(y,w)$ depending only on $\theta$ such that either 	
	$$ |E(P,Q)| \leq \alpha \left( m^{\frac{d}{2d-1}} n^{\frac{2d-2}{2d-1}} + m + n \right)$$
 for all finite $P \subseteq M^2, Q \subseteq M^d$ with $|P|=m, |Q| = n$, or there are some  $b \in M^v, c \in M^w$ such that both $\varphi(M,b)$ and $ \psi(M,c)$ are infinite and $\varphi(M,b) \times \psi(M,c) \subseteq E$.

\end{cor}
\begin{proof}
Immediate combining (1) and (2) in Theorem \ref{thm: everything o-min} (let $k'$, $\varphi$, $\psi$ be as given by (2) for $\theta(x,y;z)$, and let $\alpha$ be as given by (1) for this $k'$).
\end{proof}

\begin{rem}
The special case with $d=2$ and $E$ satisfying an additional assumption of $1$-dimensionality of its fibers was obtained independently by Basu and Raz \cite{basu2016minimal} using different methods.
\end{rem}

\bibliographystyle{acm}
\bibliography{refs}

\end{document}